\documentclass{amsart}
\usepackage{amsmath, amscd, amssymb, amsthm}
\usepackage{bbm}
\usepackage{latexsym}
\usepackage{amsfonts}
\usepackage{graphicx}
\usepackage[all,cmtip]{xy}
\usepackage[colorlinks,linkcolor=blue,breaklinks=blue,urlcolor=blue,citecolor=blue,anchorcolor=blue,pagebackref]{hyperref}%
\usepackage{geometry}
\newtheorem{lemma}{Lemma}

\newtheorem{proposition}{Proposition}

\renewcommand*\backref[1]{}
\renewcommand*\backrefalt[4]{ \ifcase #1 \or (cited on page #2) \else (cited on pages #2) \fi}

\newcommand{\be}{\begin{equation}}
\newcommand{\ee}{\end{equation}}
\newcommand{\bea}{\begin{eqnarray}}
\newcommand{\eea}{\end{eqnarray}}

\newcommand{\vs}{\vspace{0.5cm}}

\def\XXint#1#2#3{{\setbox0=\hbox{$#1{#2#3}{\int}$ }
\vcenter{\hbox{$#2#3$ }}\kern-.6\wd0}}

\begin{document}

\title[Lie algebras with abelian ideals of codimension $2$]{Hermitian geometry of Lie algebras with abelian ideals of codimension $2$}

\author{Yuqin Guo}
\address{Yuqin Guo. School of Mathematical Sciences, Chongqing Normal University, Chongqing 401331, China}
\email{{1942747285@qq.com}}\thanks{Zheng is partially supported by National Natural Science Foundations of China
with the grant No.12071050 and 12141101, Chongqing grant cstc2021ycjh-bgzxm0139, and is supported by the 111 Project D21024.}

\author{Fangyang Zheng}
\address{Fangyang Zheng. School of Mathematical Sciences, Chongqing Normal University, Chongqing 401331, China}
\email{20190045@cqnu.edu.cn; franciszheng@yahoo.com} \thanks{}


\begin{abstract}
We examine Hermitian metrics on unimodular Lie algebras which contains a $J$-invariant abelian ideal of codimension two, and give a classification for all Bismut K\"ahler-like and all Bismut torsion-parallel metrics on such Lie algebras.
\end{abstract}

\maketitle

\tableofcontents

\section{Introduction and statement of results}\label{intro}

A major difficulty in the study of differential geometry of Hermitian manifolds is the algebraic complexity of the curvature tensor, which lack of symmetry due to the presence of torsion. When the universal cover of the manifold is a Lie group, the situation is much simpler in the sense that every point looks like every other point. But even in this case the torsion and curvature of the Hermitian metric already present rich and interesting algebraic complexity, which could serve as good sample spaces as well as provide guidance towards the study of the general case. For this reason, it makes sense to study Hermitian geometry of Lie groups.

Recall that a {\em Lie-Hermitian manifold} is a compact Hermitian manifold $(M^n,g)$ whose universal cover is $(G,J,g)$, where $G$ is an even-dimensional (connected and simply-connected) Lie group, equipped with a left-invariant complex structure $J$ and a left-invariant Riemannian metric $g$ compatible with $J$. The compactness of $M$ forces $G$ to be unimodular.

In the past decades, Lie-Hermitian manifolds were studied from various angles by A. Gray, S. Salamon, L. Ugarte, A. Fino, L. Vezzoni, F. Podest\`a, D. Angella, A. Andrada, and others. There is a vast amount of literature on this topic, here we will just mention a small sample: \cite{AU}, \cite{CFGU}, \cite{EFV}, \cite{FP3}, \cite{FinoTomassini}, \cite{FinoVezzoni}, \cite{GiustiPodesta}, \cite{Salamon}, \cite{Ugarte}, \cite{WYZ}, and interested readers could start their exploration there. For more general discussions on non-K\"ahler Hermitian geometry, see for example \cite{AI}, \cite{AT}, \cite{Fu}, \cite{STW}, \cite{Tosatti} and the references therein.

The cases when $G$ is nilpotent or when $\dim G\leq 6$ are relatively well-understood, while the non-nilpotent and general dimensional cases are less clear. One exception is perhaps the case of {\em almost abelian groups}, where the Lie algebra ${\mathfrak g}$ of $G$ has an abelian ideal of codimension one. A number of interesting results were obtained about this special type, see  for instance \cite{AO}, \cite{AL}, \cite{Bock}, \cite{CM}, \cite{FP1}, \cite{FP2}, \cite{LW}, \cite{Paradiso}, \cite{Ugarte} and the references therein for more details.

Another special types of Lie-Hermitian manifolds are quotients of {\em $2$-step solvable Lie groups}. In the beautiful recent papers \cite{FSwann},\cite{FSwann2}, Freibert and Swann studied the Hermitian geometry of such Lie groups, especially on balanced and pluriclosed (namely, SKT) structures.

The purpose of this article is two fold: on one hand we would like to demonstrate the convenience of using complex coordinates in the study of Hermitian structures on Lie algebras, and we would like to illustrate this point by summarizing some of the existing results about almost abelian Lie algebras, and give hopefully somewhat simplified computation and argument there. This will be treated in \S 3, where we will collect some known results on almost abelian Lie algebras and also prove a couple of new ones, regarding {\em Bismut torsion-parallel} metrics and {\em astheno-K\"ahler} metrics. The second purpose is to push this complex coordinate computation to another very special type of Lie algebras: those which contain an abelian ideal of codimension two. Note that such a Lie algebra is always solvable of step at most $3$, but in general it will not be solvable of step $2$.

Let ${\mathfrak a}$ be an abelian ideal of codimension $2$ in a Lie algebra ${\mathfrak g}$ of real dimension $2n$. If $J$ is an integrable almost complex structure on ${\mathfrak g}$, then the $J$-invariant subspace ${\mathfrak a}_J:= {\mathfrak a}\cap J {\mathfrak a}$ has codimension either $2$ or $4$ in ${\mathfrak g}$, and it is of codimension $2$ if and only if $J{\mathfrak a}={\mathfrak a}$. In this article we will focus on this simpler case, and mimic the study on almost abelian Lie algebras to characterize some special types of Hermitian structures on such Lie algebras. This will be carried out in \S 4. The main result is Proposition \ref{prop5} below, which gives a description of all Bismut torsion-parallel metrics on such Lie algebras.

We will prove the following statements regarding almost abelian Lie algebras or Lie algebras with $J$-invariant abelian ideals of codimension $2$. For an almost abelian Lie algebra ${\mathfrak g}$ equipped with a Hermitian structure $(J,g)$, there always exists a unitary basis $\{ e_1, \ldots , e_n\}$ of the complex vector space  ${\mathfrak g}^{1,0}=\{ x-\sqrt{-1}Jx \, \mid \, x\in {\mathfrak g}\}$ so that the structure equation of $g$ takes the form
\begin{equation} \label{structure1}
\left\{ \begin{split}  d\varphi_1 = - \lambda \,\varphi_1\wedge \overline{\varphi}_1 , \hspace{5.8cm} \\ d\varphi_i = - \overline{v}_i \, \varphi_1\wedge \overline{\varphi}_1 +  \sum_{j=2}^n \overline{A_{ij}} \,(\varphi_1 + \overline{\varphi}_1)\wedge \varphi_j , \ \ \ 2\leq i\leq n. \end{split} \right.
\end{equation}
Here $\varphi$ is the unitary coframe dual to $e$, $\lambda \in {\mathbb R}$, $v \in {\mathbb C}^{n-1}$, and $A$ is a complex $(n-1)\times (n-1)$ matrix. Such a frame will be called {\em admissible} for the Hermitian almost abelian Lie algebra.

Recall that a Hermitian metric is said to be {\em Bismut torsion-parallel} ({\em BTP} in short) if its Bismut connection has parallel torsion tensor. It is called {\em Bismut K\"ahler-like} ({\em BKL} in short) if the curvature tensor of the Bismut connection obeys all K\"ahler symmetries. An equivalent form of a conjecture proposed by Angella, Otal, Ugarte, Villacampa in \cite{AOUV} states that {\em BKL} metrics are always {\em BTP}. This was confirmed by Zhao and the second named author in  \cite{ZZ-pluriclosed}.

\begin{proposition} \label{prop1}
Let $({\mathfrak g},J,g)$ be a Hermitian structure on an almost abelian Lie algebra. Then $g$ is Bismut torsion-parallel if and only if it is Bismut K\"ahler-like if and only if $A+A^{\ast}=0$ and $Av=0$ under any admissible frame.
\end{proposition}

Here and from now on $A^{\ast}$ stands for conjugate transpose $\overline{\,^t\!A}$. Note that ${\mathfrak g}$ is unimodular if and only if $\lambda + \mbox{tr}(A) + \overline{\mbox{tr}(A)} =0$. So when ${\mathfrak g}$ is unimodular, the skew-Hermitian property of $A$ implies that $\lambda =0$, and one can choose a unitary frame $e$ so that $A$ is diagonal.

Recall that  a Hermitian manifold $(M^n,g)$ is said to be {\em astheno-K\"ahler} \cite{JY}, if $\partial \overline{\partial} (\omega^{n-2})=0$, where $\omega$ is the K\"ahler form of $g$. For almost abelian Lie algebras, we have the following

\begin{proposition} \label{prop2}
Let $({\mathfrak g},J,g)$ be a Hermitian structure on a unimodular almost abelian Lie algebra. Assume that the complex dimension $n\geq 4$. Then $g$ is astheno-K\"ahler  if and only if $g$ is pluriclosed if and only if $[A^{\ast},A]=0$ and one of the eigenvalue of $A$ has real part equal to $-\frac{\lambda }{2}$ while the other eigenvalues have real part equal to $0$.
\end{proposition}

The characterization for pluriclosed metrics on almost abelian Lie algebras were obtained by Arroyo and Lafuente \cite{AL}. Note that when $n=3$ the astheno-K\"ahler and pluriclosed conditions coincide, so the above statement holds automatically.

If one drops the assumption that ${\mathfrak g}$ is unimodular, then from the proof of Proposition \ref{prop2} we know that $g$ is astheno-K\"ahler  if and only if $[A^{\ast},A]=0$ and the real part of the eigenvalues of $A$ satisfy the following condition:

{\em There exists an integer $0\leq k\leq n-1$ such that $k$ of them equal to $- \frac{(n-1-k)}{2(n-2)}\lambda$ while the other $n-1-k$ of them equal to $\frac{(k-1)}{2(n-2)}\lambda$.}

In the unimodular case, $k$ is necessarily $1$, so the real part of one eigenvalue of $A$ equals $-\frac{\lambda}{2}$ while the real part of the other eigenvalues of $A$ all equal to $0$.

We remark that, the simplicity of the structure equation (\ref{structure1}) enables one to explore Hermitian properties on almost abelian Lie algebras with relative ease, and the above two propositions were intended to serve as an illustration.

Next, let ${\mathfrak g}$ be a Lie algebra equipped with a Hermitian structure $(J,g)$, and let ${\mathfrak a}\subseteq {\mathfrak g}$ be a $J$-invariant abelian ideal of codimension $2$. Then there always exists unitary frame $e$ under which the structure equation becomes
\begin{equation} \label{structure2}
\left\{ \begin{split} d\varphi_1 \, = \, -\lambda \varphi_1\overline{\varphi}_1 , \hspace{4.1cm} \\
d\varphi \, = \, - \varphi_1\overline{\varphi}_1 \overline{v} -\varphi_1\,^t\!X \varphi + \overline{\varphi}_1 \overline{Y} \varphi - \varphi_1 \overline{Z} \,\overline{\varphi}. \end{split} \right.
\end{equation}
Here we wrote $\varphi = \,^t\!(\varphi_2, \ldots , \varphi_n)$ as a column vector, while $\lambda \geq 0$, $v\in {\mathbb C}^{n-1}$, and $X$, $Y$, $Z$ are complex $(n-1)\times (n-1)$ matrices satisfying
\begin{equation} \label{XYZ}
\left\{ \begin{split} \lambda (X^{\ast}\!+Y)+ [X^{\ast} ,Y]  -  Z\overline{Z} \, = \, 0, \\
\lambda Z - ( Z \,^t\!X + Y Z )\, =\, 0. \hspace{1.2cm} \end{split} \right.
\end{equation}
Such a unitary frame $e$ will be called an {\em admissible frame}.  Note that the algebraic data involves three matrices satisfying a system of matrix equations and are tangled up, unlike the almost abelian case. Also, when $\lambda \neq 0$, ${\mathfrak g}$ is not $2$-step solvable in general, although it is always $3$-step solvable.

\begin{proposition} \label{prop3}
Let $({\mathfrak g},J,g)$ be a Lie algebra with Hermitian structure and let ${\mathfrak a}\subseteq {\mathfrak g}$ be a $J$-invariant abelian ideal of codimension $2$. Then under any admissible frame,
\begin{enumerate}
\item ${\mathfrak g}$ is unimodular if and only if $\, \lambda - \mbox{tr}(X)+\mbox{tr}(Y)=0$,
\item $g$ is balanced if and only if $\, \mbox{tr}(X)=\mbox{tr}(Y)$ and $v=0$,
\item $g$ is K\"ahler if and only if $v=0$, $\,^t\!Z=Z$, and $X=Y$. When ${\mathfrak g}$ is unimodular, $g$ is K\"ahler if and only if $\lambda =0$, $v=0$, $Z=0$, and $X=Y$ is normal. In this case there is a unitary frame under which the structure equation becomes
    \begin{equation*}
    d\varphi_1= 0, \ \ \ d\varphi_i = (\overline{a}_i\overline{\varphi}_1-a_i\varphi_1)\varphi_i, \ \ \ \ 2\leq i\leq n.
    \end{equation*}
    where $a_2, \ldots , a_n$ are arbitrary complex numbers.
\item $g$ is pluriclosed (SKT) if and only if it satisfies (\ref{XYZ}) and the equation below:
\begin{equation} \label{SKT-XYZ}
 XX^{\ast}-YX^{\ast} + \,^t\!Z \overline{Z} - Z \overline{Z}    + Y^{\ast} Y - Y^{\ast} X + \lambda (Y-X) = 0.
 \end{equation}
\end{enumerate}
\end{proposition}

From items (ii) and (iv) above, one can write down explicit examples of balanced metrics or pluriclosed metrics on such ${\mathfrak g}$. See \S 4 for more details. Similar to the almost abelian case, one could also analyze the behavior of the Chern and Bismut curvature of $g$. In particular, one has the following

\begin{proposition} \label{prop4}
Let $({\mathfrak g},J,g)$ be a Lie algebra with Hermitian structure and ${\mathfrak a}$ be a $J$-invariant abelian ideal of codimension $2$. Then the Chern curvature satisfies the following
\begin{enumerate}
\item $g$ is Chern flat  if and only if $\lambda =0$, $v=0$, $Z=0$, $[Y,Y^{\ast}]=0$, and $[Y,X^{\ast}]=0$. In this case one can choose unitary frame so that $Y$ is diagonal and $X$ is block-diagonal in the form (\ref{Chernflat}) below;
 \item The rank of the first Chern Ricci $Ric^{(1)}$ is at most $1$, and it has the same sign with the Chern scalar curvature $s$. In particular, $Ric^{(1)}=0\  \Longleftrightarrow \  s=0\  \Longleftrightarrow  \ \lambda =0  \ \mbox{or} \  \lambda =-\mbox{Re}\{ \mbox{tr}(Y)\}$;
    \item If the second Chern Ricci $Ric^{(2)}\geq 0$, then $R=0$;
    \item For the (first) Bismut Ricci curvature $Ric^b$, its $(2,0)$-part has rank $2$ or $0$, and its $(1,1)$-part has rank at most $1$ if $v=0$ and has rank at most $3$ in general. $g$ is Bismut Ricci flat  if and only if $Xv=0$, $Y^{\ast}v+\,^t\!Z\overline{v}=0$, and $s^b=0$, where the Bismut scalar curvature $s^b$ is given by (\ref{Bismutscalar}) below. In particular, $g$ is Bismut Ricci flat when $v=0$ and $\lambda =0$, or when $v=0$ and $\lambda = \mbox{Re}\{ \mbox{tr}(X) - 2 \mbox{tr}(Y)\}$.
    \end{enumerate}
\end{proposition}

In the Chern flat case, since $Y$ is normal and $[X^{\ast},Y]=0$, we can choose a unitary frame $e$ so that $Y$ is diagonal and $X$ is block-diagonal in the following way:
\begin{equation} \label{Chernflat}
Y = \left[ \begin{array}{ccc} \lambda_1 I_{n_1} & & \\ & \ddots & \\ && \lambda_rI_{n_r} \end{array} \right], \ \ \ \ \ \ X = \left[ \begin{array}{ccc} X_1 & & \\ & \ddots & \\ && X_r \end{array} \right],
\end{equation}
where $n_1+ \cdots +n_r=n-1$ and $\lambda_1, \ldots , \lambda_r$ are distinct, while each $X_i$ is an arbitrary $n_i\times n_i$ matrix. The Bismut scalar curvature in (iv) above is given by
\begin{equation} \label{Bismutscalar}
s^b \, = \, -2|v|^2 - \lambda \left( 2\lambda + 2\mbox{tr}(Y) + 2\overline{\mbox{tr}(Y)} -\mbox{tr}(X) -\overline{\mbox{tr}(X)}\right) .
\end{equation}
A Hermitian metric $g$ with zero Bismut (first) Ricci curvature is said to be {\em Calabi-Yau with torsion}, or {\em CYT} in short. In \cite[Theorem 3.1]{FP2}, Fino and Paradiso characterized CYT metrics on almost abelian Lie algebras, and item (iv) above is just analogous to their result.

Also analogous to the almost abelian case, for ${\mathfrak g}$ with $J$-invariant abelian ideal of codimension $2$, one could also ask when will the metric be Bismut torsion-parallel or Bismut K\"ahler-like. The following is the main result of this article:

\begin{proposition} \label{prop5}
Let $({\mathfrak g},J,g)$ be a unimodular Lie algebra with Hermitian structure and ${\mathfrak a}$ be a $J$-invariant abelian ideal of codimension $2$. Then $g$ is Bismut K\"ahler-like if and only if there exists a unitary frame under which the structure equation is given by (\ref{BTPv1}) below. $g$ is Bismut torsion-parallel if and only if there exists a unitary frame under which the structure equation is given by (\ref{BTPv1}), or (\ref{BTPv2}), or (\ref{BTPv0}).
\end{proposition}

\begin{equation} \label{BTPv1}
 \left\{ \begin{split} d\varphi_1 = 0,  \hspace{4.2cm}\\
d\varphi_2 = - v_2 \varphi_1\overline{\varphi}_1, \hspace{3cm}\\
d\varphi_i = (\overline{a}_i\overline{\varphi}_1 - a_i \varphi_1)\varphi_i, \ \ \ 3\leq i\leq n;
\end{split} \right.   \ \ \ \ \ \ \ \ \ \ \ \ \ \ \
\end{equation}
where $v_2>0$ and $a_3, \ldots , a_n$ are complex numbers.

\begin{equation} \label{BTPv2}
 \left\{ \begin{split} d\varphi_1 =d\varphi_3 =  0, \hspace{3.15cm} \\
d\varphi_2 = - v_2 \varphi_1\overline{\varphi}_1 + p \,(\overline{\varphi}_1\varphi_3 - \varphi_1 \overline{\varphi}_3),  \hspace{0.1cm} \\
d\varphi_i = (\overline{a}_i\overline{\varphi}_1 - a_i \varphi_1)\varphi_i, \ \ \ 4\leq i\leq n.
\end{split} \right.
\end{equation}
where $v_2>0$, $p>0$, and $a_4, \ldots , a_n$ are complex numbers.

\begin{equation}  \label{BTPv0}
\left\{ \begin{split} d\varphi_1 = 0, \hspace{6.4cm} \\
d\varphi_i =  \lambda_i \overline{\varphi}_1\varphi_i - \lambda_i \varphi_1 \sum_{j=2}^{r+1} \overline{W_{ij}}\varphi_j ,  \ \ \ \ \ 2\leq i\leq r+1, \hspace{0.1cm} \\
d \varphi_i = 0, \ \ \ \ \ \ \ \ \  r+2\leq i\leq 2r+1,   \hspace{2.45cm}\\
d\varphi_i = (\overline{a}_i\overline{\varphi}_1 - a_i \varphi_1)\varphi_i, \ \ \ \ \ 2r+2\leq i\leq n.   \hspace{1.2cm}
\end{split} \right.
\end{equation}
where $r$ is an integer with $1\leq r \leq \frac{n-2}{2}$, each $a_i$ is a complex number,  $S=\mbox{diag} \{ \lambda_2 , \ldots , \lambda_{r+1}\} > 0$ and $W$ is any unitary symmetric $r\times r$ matrix satisfying $WS=SW$.

Note that (\ref{BTPv1}) and (\ref{BTPv2}) are not balanced, while (\ref{BTPv0}) is balanced. Analogous to the almost abelian case, one could show that, {\em for any unimodular Lie algebra with $J$-invariant abelian ideal of codimension $2$,  if the metric is Chern K\"ahler-like, then it must be Chern flat.} So again there is no non-trivial example of compact Chern K\"ahler-like manifolds amongst  this type of Lie algebras. Also analogous to the almost abelian case, for unimodular ${\mathfrak g}$ with $J$-invariant abelian ideal of codimension $2$, when the complex dimension $n\geq 4$, then one can write down explicitly the matrix equation for the metric to be astheno-K\"ahler, in the form $P=\mu I$ where $P$ is a matrix involving $X$, $Y$, $Z$ and $\mu$ is a constant. The computation is strictly analogous but slightly more tedious, and we will skip it here.

The article is organized as follows. In \S 2 we will set up notations and collect some basic formula for Lie-Hermitian manifolds. In \S 3 we will recall some known properties about almost abelian groups and also prove a couple of new ones: Propositions \ref{prop1} and \ref{prop2} stated above. In \S 4, we will turn our attention to Lie algebras with $J$-invariant abelian ideals of codimension $2$, and prove a number of properties including Propositions \ref{prop3} - \ref{prop5} stated above.

\vspace{0.3cm}

\section{Preliminary on Lie-Hermitian manifolds}

In this section, we will collect some basic formula for Lie-Hermitian manifolds that will be used later. We will follow the notations in \cite{VYZ}.

Let $(M^n,g)$ be a Lie-Hermitian manifold, with universal cover $(G,J,g)$ where $G$ is an even-dimensional Lie group and $J$ is a left-invariant complex structure on $G$, while $g$ is a left-invariant Riemannian metric on $G$ compatible with $J$. As we mention before, the compactness of $M$ forces $G$ to be unimodular.

Denote by ${\mathfrak g}$ the Lie algebra of $G$, and use the same letter $J$ or $g=\langle , \rangle $ to denote respectively the almost complex structure or inner product on ${\mathfrak g}$ corresponding to that of $G$. As is well known, the integrability of $J$ is characterized by the property
\begin{equation} \label{int}
[x,y] - [Jx,Jy] + J[Jx,y] + J[x,Jy] = 0, \ \ \ \ \forall \ x,y\in {\mathfrak g}.
\end{equation}
Denote by ${\mathfrak g}^{\mathbb C}$ the complexification of ${\mathfrak g}$, and by ${\mathfrak g}^{1,0}= \{ x-\sqrt{-1}Jx \mid x \in {\mathfrak g}\} \subseteq {\mathfrak g}^{\mathbb C}$. The condition (\ref{int}) simply says that ${\mathfrak g}^{1,0}$ is a complex Lie subalgebra of ${\mathfrak g}^{\mathbb C}$. Extend $g=\langle , \rangle $ bi-linearly over ${\mathbb C}$, and let $e=\{ e_1, \ldots , e_n\}$ be a unitary basis of ${\mathfrak g}^{1,0}$. Following \cite{VYZ}, we will use
\begin{equation*} \label{CandD}
C^j_{ik} = \langle [e_i,e_k],\overline{e}_j \rangle, \ \ \ \ \ \  D^j_{ik} = \langle [\overline{e}_j, e_k] , e_i \rangle
\end{equation*}
to denote the structure constants, or equivalently, under the unitary frame $e$ we have
\begin{equation} \label{CandD2}
[e_i,e_j] = \sum_k C^k_{ij}e_k, \ \ \ \ \ [e_i, \overline{e}_j] = \sum_k \big( \overline{D^i_{kj}} e_k - D^j_{ki} \overline{e}_k \big) .
\end{equation}
Note that ${\mathfrak g}$ is unimodular if and only if $\mbox{tr}(ad_x)=0$ for any $x\in {\mathfrak g}$, which is equivalent to
\begin{equation} \label{unimo}
{\mathfrak g} \ \, \mbox{is unimodular}  \ \ \Longleftrightarrow  \ \ \sum_r \big( C^r_{ri} + D^r_{ri}\big) =0 , \, \ \forall \ i.
\end{equation}

We should note that two Lie-Hermitian manifolds could be holomorphically isometric to each other, while their universal covers are not isomorphic as Lie groups. For instance, there are Lie-Hermitian manifold $M^n$ with non-abelian $G$ whose metric turns out to be K\"ahler and flat, thus forcing $M^n$ to be a quotient of the abelian group ${\mathbb C}^n$. In general it would be an interesting question to determine when will two non-isomorphic Lie groups be holomorphically isometric as Hermitian manifolds, but since our goal is primarily to find examples of Hermitian manifolds satisfying special geometric conditions, we will ignore this more subtle issue of group isomorphism here.

Next, let us denote by $\nabla$ the Chern connection, and by $T$, $R$ its torsion and curvature tensor. Then one has
\begin{equation*} \label{Gamma}
\nabla e_i = \sum_j \theta_{ij} e_j, \ \ \ \ \theta_{ij} = \sum_k \big( \Gamma^j_{ik} \varphi_k - \overline{\Gamma^i_{jk}}\, \overline{\varphi}_k \big), \ \ \ \ \ \Gamma^j_{ik} = D^j_{ik},
\end{equation*}
where $\{ \varphi_1, \ldots , \varphi_n\}$ is  the coframe dual to $e$. The torsion tensor $T$ of $\nabla$ has components
\begin{equation} \label{torsion}
T( e_i, \overline{e}_j)=0, \ \ \ \ T(e_i,e_j)  = \sum_k T^k_{ij}e_k, \ \ \ \  \ \ T^j_{ik}= - C^j_{ik} -  D^j_{ik} + D^j_{ki}.
\end{equation}
The only possibly non-zero components of $R$ are $R_{i\bar{j}k\bar{\ell}}= \langle R_{e_i\bar{e}_j} e_k, \bar{e}_{\ell} \rangle $, which are equal to
\begin{equation} \label{curvature}
R_{i\bar{j}k\bar{\ell}} = \sum_r \big( D^r_{ki}\overline{D^r_{\ell j}}  - D^{\ell}_{ri}\overline{D^k_{r j}} - D^j_{ri}\overline{D^k_{\ell r}} - \overline{D^i_{r j}} D^{\ell}_{kr} \big).
\end{equation}
Denote by $\omega = \sqrt{-1}\sum_i \varphi_i \wedge \overline{\varphi}_i$ the K\"ahler form of the metric.  Recall that {\em Gauduchon's torsion $1$-form} $\eta$ (\cite{Gau84}) is defined by $\partial \omega^{n-1} = - \eta \wedge \omega^{n-1}$. Its components are
\begin{equation} \label{eta}
\eta = \sum_i \eta_i \varphi_i, \ \ \ \ \ \ \eta_i = \sum_r T^r_{ri} = \sum_r D^r_{ir},
\end{equation}
where the last equality holds when ${\mathfrak g}$ is unimodular (\ref{unimo}). The structure equation takes the form:
\begin{equation} \label{structure}
d\varphi_i = -\frac{1}{2} \sum_{j,k} C^i_{jk} \,\varphi_j\wedge \varphi_k - \sum_{j,k} \overline{D^j_{ik}} \,\varphi_j \wedge \overline{\varphi}_k
\end{equation}
Differentiate the above, we get the  first Bianchi identity, which is equivalent to the Jacobi identity in this case:
\begin{eqnarray}
&& \sum_r \big( C^r_{ij}C^{\ell}_{rk} + C^r_{jk}C^{\ell}_{ri} + C^r_{ki}C^{\ell}_{rj} \big) \ = \ 0, \label{CC} \\
&& \sum_r \big( C^r_{ik}D^{\ell}_{jr} + D^r_{ji}D^{\ell}_{rk} - D^r_{jk}D^{\ell}_{ri} \big) \ = \ 0, \label{CD} \\
&& \sum_r \big( C^r_{ik}\overline{D^r_{j \ell}}  - C^j_{rk}\overline{D^i_{r \ell}} + C^j_{ri}\overline{D^k_{r \ell}} -  D^{\ell}_{ri}\overline{D^k_{j r}} +  D^{\ell}_{rk}\overline{D^i_{jr}}  \big) \ = \ 0 . \label{CDbar}
\end{eqnarray}
Next, denote by  $\nabla^b$ the Bismut connection of $g$, and let $\Gamma^{bj}_{ik} = \langle \nabla^b_{e_k} e_i, \overline{e}_j\rangle$ be the connection coefficients under the frame $e$. We have
\begin{equation*} \label{Bismut}
\Gamma^{bj}_{ik} = \Gamma^j_{ik} + T^j_{ik} = - C^j_{ik} + D^j_{ki}.
\end{equation*}
More generally, for any $t\in {\mathbb R}$, denote by $\nabla^{(t)}= (1-\frac{t}{2})\nabla + \frac{t}{2}\nabla^b$ the {\em  $t$-Gauduchon connection} (\cite{Gau97}). Its connection coefficients are given by
\begin{equation*} \label{t-Gauduchon}
\Gamma^{(t)j}_{ik} = \Gamma^j_{ik} + \frac{t}{2}T^j_{ik} = (1-\frac{t}{2})D^j_{ik} + \frac{t}{2} \big( - C^j_{ik} + D^j_{ki}\big).
\end{equation*}
The following observation should be well-known to experts, and we give a proof here for readers' convenience.
\begin{lemma} \label{lemma1}
Let ${\mathfrak g}$ be a unimodular Lie algebra with a Hermitian structure $(J,g)$, then $g$ is always Gauduchon, namely, $\partial \overline{\partial} \omega^{n-1} = 0$. Here $\omega$ is the K\"ahler form of $g$.
\end{lemma}

\begin{proof}
By the defining equation of $\eta$, one gets
$$\partial \overline{\partial} \omega^{n-1} = (\overline{\partial}\eta + \eta \wedge \overline{\eta})\wedge \omega^{n-1}.
$$
So $g$ is Gauduchon if and only if $\chi = |\eta|^2$, where $|\eta|^2=\sum_i |\eta_i|^2$ and $\chi =\sum_i \eta_{i,\bar{i}}$ under any unitary frame. Here and from now on the index after comma stands for covariant derivatives with respect to the Chern connection $\nabla$. Since $\Gamma^j_{ik}=D^j_{ik}$ and ${\mathfrak g}$ is unimodular, by (\ref{eta}) we have
$$ \chi = \sum_i \eta_{i,\bar{i}} = \sum_{i,r} \eta_r \overline{\Gamma^i_{ri}} = \sum_r \eta_r \overline{\eta_r} = |\eta|^2. $$
Therefore $g$ is always Gauduchon.
\end{proof}

Now let us discuss the geometric meaning of the constants $C$ and $D$. Note that $J$ makes ${\mathfrak g}$ a complex Lie algebra when and only when $D=0$, so the tensor $D$ in a way measures how far $({\mathfrak g},J)$ is deviated from being a complex Lie algebra. When $D=0$, we know from (\ref{curvature}) that $g$ is Chern flat. Note that the converse of this not true: if the Hermitian manifold $(G,J,g)$ is Chern flat, then $D$ might not be identically zero.

For the tensor $C$,  $C=0$ means $[Jx,Jy]=[x,y]$ for any $x,y\in {\mathfrak g}$. In this case $J$ is called an {\em abelian complex structure,} representing a rather special type of Lie-Hermitian manifolds. The condition $D=-C$ is equivalent to $g$ being a {\em bi-invariant metric} on $G$, and in this case the Bismut connection $\nabla^b$ (\cite{Bismut}) has coefficients $\Gamma^{b}=0$, hence $g$ is Bismut flat. A beautiful recent result of Lafuente and Stanfield \cite{LS}, settling a question raised in \cite{YZ1}, states that for any $t\neq 0,2$ if a Hermitian manifold has flat $t$-Gauduchon connection $\nabla^{(t)}$, then $g$ must be K\"ahler. In the special case of Lie-Hermitian manifolds, their result implies that

\begin{proposition}[\cite{LS}] \label{prop6}
For any $t\neq 0,2$, if $({\mathfrak g}, J,g)$ satisfies $\Gamma^{(t)} = 0$, then $g$ is K\"ahler, namely, $T=0$. In other words, if $t\neq 0,2$,  $D^j_{ik}=-D^j_{ki}$, and $C=\frac{2(1-t)}{t}D$, then $C=D=0$ thus ${\mathfrak g}$ is abelian.
\end{proposition}

Note that for Lie-Hermitian manifolds, $T=0$ does not imply ${\mathfrak g}$ is abelian in general, as we mentioned before. But in this particular case, since $T=-C-D+\,^t\!D= -C-2D$ as $\,^t\!D=-D$, so when $T=0$ we get $C= -2D=-2(1-\frac{1}{t})D$, which leads to $D=0$ hence $C=0$. We remark that, the result of Lafuente and Stanfield is highly non-trivial even in the special case of Lie-Hermitian manifolds.

 Recall that a Hermitian metric $g$ is called {\em pluriclosed} (also called {\em strong K\"ahler with torsion}, or {\em SKT} in short in some literature), if its K\"ahler form $\omega$ satisfies $\partial \overline{\partial} \omega =0$. For Lie-Hermitian manifolds, one has

\begin{lemma} \label{lemma2}
Let $({\mathfrak g}, J,g)$ be any Lie algebra with a Hermitian structure. It holds that
\begin{equation*}
\sqrt{-1}\partial \overline{\partial} \omega = \sum_{i,j,k,\ell } \sum_r \left( -\frac{1}{2} T^r_{ik}  \overline{C^r_{j\ell }} - T^j_{ir} \overline{D^k_{r\ell }} + T^j_{kr} \overline{D^i_{r\ell}} \right) \varphi_i\varphi_k \overline{\varphi}_j \overline{\varphi}_{\ell}.
\end{equation*}
In particular, $g$ is pluriclosed if and only if
\begin{equation} \label{SKT}
\sum_r \left( - T^r_{ik}  \overline{C^r_{j\ell }} - T^j_{ir} \overline{D^k_{r\ell }} + T^j_{kr} \overline{D^i_{r\ell}}  +  T^{\ell}_{ir} \overline{D^k_{rj }} - T^{\ell}_{kr} \overline{D^i_{rj}}  \right) = 0
\end{equation}
for any $1\leq i<k \leq n$ and any $1\leq j< \ell \leq n$.
\end{lemma}

\begin{proof}
From the structure equation (\ref{structure}), we have
\begin{eqnarray*}
\sqrt{-1}\partial \omega  & = & - \sum_{i,j,k} T^j_{ik} \varphi_i\varphi_k\overline{\varphi}_j \\
\sqrt{-1}\partial \overline{\partial} \omega & = & - d ( \sqrt{-1}  \partial \omega) \ = \ \sum_{i,j,k} T^j_{ik} \left( (d(\varphi_i\varphi_k))^{2,1} \overline{\varphi}_j   + \varphi_i\varphi_k \overline{(d\varphi_j)^{2,0}} \right) \\
& = & \sum_{i,j,k} T^j_{ik} \left( \sum_{p,q} \big( \overline{D^p_{iq}} \varphi_p\varphi_k + \overline{D^p_{kq}} \varphi_i \varphi_p \big) \overline{\varphi}_q  \overline{\varphi}_j -\sum_{p,q} \frac{1}{2} \overline{C^j_{pq}} \varphi_i\varphi_k \overline{\varphi}_p \overline{\varphi}_q \right) \\
& = & \sum_{i,j,k,\ell } \sum_r \left( -\frac{1}{2} T^r_{ik}  \overline{C^r_{j\ell }} - T^j_{ir} \overline{D^k_{r\ell }} + T^j_{kr} \overline{D^i_{r\ell}} \right) \varphi_i\varphi_k \overline{\varphi}_j \overline{\varphi}_{\ell} \\
& = &  \frac{1}{2} \sum_{i,j,k,\ell } \sum_r \left( - T^r_{ik}  \overline{C^r_{j\ell }} - T^j_{ir} \overline{D^k_{r\ell }} + T^j_{kr} \overline{D^i_{r\ell}}  +  T^{\ell}_{ir} \overline{D^k_{rj }} - T^{\ell}_{kr} \overline{D^i_{rj}}  \right) \varphi_i\varphi_k \overline{\varphi}_j \overline{\varphi}_{\ell}.
\end{eqnarray*}
This completes the proof of the lemma.
\end{proof}

Let us denote by $\zeta_i = \sum_r D^r_{ri}$ under any unitary frame $e$. When ${\mathfrak g}$ is unimodular, we also have $\zeta_i=-\sum_r C^r_{ri}$. Define $\zeta = \sum_i \zeta_i \varphi_i$, where $\varphi$ is the coframe dual to $e$, then clearly $\zeta$ is independent of the choice of $e$, hence is a globally defined left-invariant $(1,0)$-form on $G$. By the structure equation (\ref{structure}), a straight-forward computation leads to

\begin{lemma} \label{lemma3}
Let $(G,J,g)$ be the universal cover of a Lie-Hermitian manifold. Then one has
\begin{equation*}
d (\varphi_1 \varphi_2 \cdots \varphi_n) = \overline{\zeta} \wedge \varphi_1 \varphi_2 \cdots \varphi_n.
\end{equation*}
where $\varphi$ is a unitary coframe of left-invariant $(1,0)$-forms. So when $\zeta =0$, the global $(n,0)$-form $\varphi_1 \varphi_2 \cdots \varphi_n$ is a nowhere zero holomorphic $n$-form, hence $g$ is Chern Ricci flat.
\end{lemma}

Conversely, if $G$ admits a global nowhere zero holomorphic $n$-form, then $\zeta = \partial f$ for some smooth function $f$ on $G$, namely, $\zeta $ is $\partial$-exact.

By taking trace of the Chern curvature tensor $R$ (see (\ref{curvature})) in three different ways, one gets the {\em first, second, and third Chern Ricci} curvature tensors:
\begin{eqnarray}
Ric^{(1)}_{i\bar{j}} & = & \sum_r R_{i\bar{j}r\bar{r}} \ = \ -\sum_r \big( \zeta_r \overline{D^i_{rj}} + \overline{\zeta}_r D^j_{ri} \big) ,  \label{Ric1}\\
Ric^{(2)}_{i\bar{j}} & = & \sum_r R_{r\bar{r}i\bar{j}} \ = \ \sum_{r,s}\big( D^r_{is}\overline{D^r_{js}} - D^j_{rs}\overline{D^i_{rs}} \big) - \sum_r \big( \eta_r \overline{D^i_{jr}} + \overline{\eta}_r D^j_{ir}\big), \label{Ric2} \\
Ric^{(3)}_{i\bar{j}} & = & \sum_r R_{r\bar{j}i\bar{r}} \ = \ - \sum_{r,s} D^j_{rs}\overline{D^i_{sr}}  - \sum_r  \eta_r \overline{D^i_{rj}} . \label{Ric3}
\end{eqnarray}
From the first equation, we could also see that $\zeta =0$ implies $Ric^{(1)}=0$. In the second and third line, we assumed that ${\mathfrak g}$ is unimodular, so $\eta_i=\sum_r D^r_{ir}$. By taking trace again on these Ricci tensors, one gets the {\em Chern scalar curvature} $s$ and {\em altered scalar curvature} $\hat{s}$:
\begin{eqnarray}
s & = &  \sum_i Ric^{(1)}_{i\bar{i}} \  = \ \sum_i Ric^{(2)}_{i\bar{i}} \ = \ - \sum_i \big( \zeta_i \overline{\eta}_i + \overline{\zeta}_i \eta_i \big), \label{s} \\
\hat{s} & = & \sum_i Ric^{(3)}_{i\bar{i}} \  = \ - \sum_{r,s,t} D^t_{rs} \overline{D^t_{sr}} - |\eta|^2. \label{shat}
\end{eqnarray}
Note that on any Hermitian manifolds, $\hat{s}=s-\chi$ always holds, and $\hat{s}$ and $\chi$ are both real. For Lie-Hermitian manifolds, we have $|\eta|^2=\chi$, and by (\ref{CDbar}) one gets
$$ \sum_{r,s} D^j_{rs} \overline{D^i_{sr}} = \sum_r \big( \zeta_r \overline{D^i_{rj}} + \overline{\zeta}_r D^j_{ri} \big) , \ \ \ \ \ \forall \ i,j. $$
Thus $\hat{s}=s-\chi$ holds, and $ s = - \sum_{r,s,t} D^t_{rs} \overline{D^t_{sr}}$. This last equality leads to the following

\begin{lemma} \label{lemma4}
If a Lie-Hermitian manifold $M^n$ has symmetric $D$ tensor, namely,  $D^j_{ik}=D^j_{ki}$ for any $i,j,k$,  then $s= - \sum_{r,s,t} |D^t_{rs}|^2 \leq 0$, and $s<0$ unless $D=0$. Similarly, if $D$ is skew-symmetric: $D^j_{ik}=-D^j_{ki}$, then $s= \sum_{r,s,t} |D^t_{rs}|^2 \geq 0$, and $s>0$ unless $D=0$. Note that when $s>0$ the Kodaira dimension of $M^n$ is necessarily $-\infty$.
\end{lemma}

Finally, let us close this section by recalling the covariant differentiation formula for Chern torsion:
\begin{eqnarray}
T^j_{ik,\ell} & = &  \sum_r \left( - T^j_{rk} \Gamma^r_{i\ell} -  T^j_{ir} \Gamma^r_{k\ell} + T^r_{ik} \Gamma^j_{r\ell}   \right)   \label{Tderivative}\\
T^j_{ik,\bar{\ell} } & = & \sum_r \left(  T^j_{rk} \overline{\Gamma^i_{r\ell}} +  T^j_{ir} \overline{\Gamma^k_{r\ell}}  - T^r_{ik} \overline{\Gamma^r_{j\ell}}   \right)  \label{Tderivativebar}
\end{eqnarray}
for any $1\leq i,j,k,\ell \leq n$, where index after comma stands for covariant derivatives with respect to the Chern connection $\nabla$. Of course when $\nabla$ is replaced by another Hermitian connection, then the same differentiation formula holds when $\Gamma$ is replaced by the connection coefficients of that connection.

\vspace{0.3cm}

\section{Almost abelian Lie algebras}

In this section let us specialize to Lie-Hermitian manifolds where ${\mathfrak g}$ is an {\em almost abelian Lie algebra}, namely, when ${\mathfrak g}$ is non-abelian but contains an abelian ideal of codimension one.

An equivalent definition for this is a non-abelian Lie algebra ${\mathfrak g}$ which contains an abelian subalgebra of codimension one. In this case, ${\mathfrak g}$ always contains an abelian ideal of codimension one by \cite[Prop. 3.1]{BC}.

Let ${\mathfrak g}$ be an almost abelian Lie algebra, with ${\mathfrak a} \subseteq {\mathfrak g}$ an abelian ideal of codimension one. Assume that $(J,g)$ is an Hermitian structure on ${\mathfrak g}$, with $\dim {\mathfrak g}=2n$. Since  ${\mathfrak a}_J:={\mathfrak a}\cap J {\mathfrak a}$ is the intersection of two hyperplanes and is $J$-invariant, it is necessarily of codimension $2$ in ${\mathfrak g}$. Thus we can choose a unitary basis $\{ e_1, \ldots , e_n\}$ of ${\mathfrak g}^{1,0}$ so that ${\mathfrak a}$ is spanned by
$$ {\mathfrak a} = \mbox{span}_{\mathbb R}\{ \sqrt{-1}(e_1-\overline{e}_1); \ (e_i+\overline{e}_i), \, \sqrt{-1}(e_i-\overline{e}_i); \ 2\leq i\leq n \} .$$
We will call such a unitary basis an {\em admissible frame}, and we will choose $e$ to be admissible from now on. Since ${\mathfrak a}$ is abelian, we get
$$ [e_i, e_j]= [ e_i, \overline{e}_j] = [e_1-\overline{e}_1, e_i]= 0, \ \ \ \ \forall \ 2\leq i,j\leq n.$$
From this and (\ref{CandD2}) we deduce that
\begin{equation*}
C^{\ast}_{ij} = D^j_{\ast i} = D^1_{\ast i} = C^{\ast}_{1i}+\overline{D^i_{\ast 1}} =0, \ \ \ \ \forall \ 2\leq i,j\leq n.
\end{equation*}
Since ${\mathfrak a}$ is an ideal, $[e_1+\overline{e}_1, {\mathfrak a} ] \subseteq {\mathfrak a}$, which leads to $C^1_{1i}=0$ and $D^1_{11} = \overline{D^1_{11}}$. Putting these together, we know that the only possibly non-trivial components of $C$ and $D$ are
\begin{equation} \label{CandD-aala}
D^1_{11}=\lambda \in {\mathbb R}, \ \ \ D^1_{i1}=v_i \in {\mathbb C}, \ \ \ D^j_{i1}=A_{ij}, \ \ \ C^{j}_{1i} = - \overline{A_{ji}}, \ \ \ \ \ \  2\leq i,j\leq n.
\end{equation}
So the only algebraic data for the structure constants are $\lambda \in {\mathbb R}$, a complex column vector $v\in {\mathbb C}^{n-1}$, and a complex $(n-1)\times (n-1)$ matrix $A=(A_{ij})_{2\leq i,j\leq n}$. In terms of the dual coframe $\varphi$, the structure equation (\ref{structure}) now takes the following form:

\begin{lemma} \label{lemma5}
Let ${\mathfrak g}$ be an almost abelian Lie algebra with a Hermitian structure $(J,g)$. Then there exists a unitary coframe $\varphi$ under which the structure equation is
\begin{equation*}
\left\{ \begin{split}  d\varphi_1 = - \lambda \,\varphi_1\wedge \overline{\varphi}_1 , \hspace{5.8cm} \\ d\varphi_i = - \overline{v}_i \, \varphi_1\wedge \overline{\varphi}_1 +  \sum_{j=2}^n \overline{A_{ij}} \,(\varphi_1 + \overline{\varphi}_1)\wedge \varphi_j , \ \ \ 2\leq i\leq n. \end{split} \right.
\end{equation*}
\end{lemma}

This is well-known to experts studying almost abelian Lie algebras and is stated as (\ref{structure1}) in \S 1.
Clearly, $d^2\varphi_1=d^2\varphi_i=0$, so (\ref{CC}) - (\ref{CDbar}) do not impose further restriction on $\lambda$, $v$, or $A$. These three terms are not all zero, as ${\mathfrak g}$ is assumed to be non-abelian.

The following properties about Hermitian structures on almost abelian Lie algebras are discovered by various people and we summarize into a proposition, and include a proof here for readers' convenience. Item (iii) was discovered by Lauret and Will \cite{LW}, item (iv) was studied extensively by Fino and Paradiso \cite{FP1}, \cite{FP2}, and item (v) was proved by Arroyo and Lafuente \cite{AL}.

\begin{proposition} \label{propAL}
Let ${\mathfrak g}$ be an almost abelian Lie algebra with Hermitian structure $(J,g)$. Then there exists unitary coframe with structure equation (\ref{structure1}) and the following holds:
\begin{enumerate}
\item  ${\mathfrak g}$ is nilpotent $\ \Longleftrightarrow \ $ $\lambda =0$ and $A^{n-1}=0$;
\item   ${\mathfrak g}$ is unimodular $\ \Longleftrightarrow \ $ $\lambda + \mbox{tr}(A) + \overline{\mbox{tr}(A)} = 0$;
\item $g$ is K\"ahler  $\ \Longleftrightarrow \ $ $v=0$ and $A+A^{\ast}=0$.
\item $g$ is balanced  $\ \Longleftrightarrow \ $ $v=0$ and $\mbox{tr}(A) + \overline{\mbox{tr}(A)}=0$.
\item $g$ is pluriclosed  $\ \Longleftrightarrow \ $ $\lambda (A+A^{\ast}) + (A+A^{\ast})^2 + [A^{\ast}, A]=0$ \
$\ \Longleftrightarrow \ $ $[A,A^{\ast}]=0$ and the real part of each eigenvalue of $A$ is either $0$ or $-\frac{\lambda}{2}$.
 \end{enumerate}
\end{proposition}

\begin{proof}
Let $e$ be an admissible frame. The components of $C$ and $D$ are given by the real constant $\lambda$, complex column vector $v$, and complex square matrix $A$, with structure equation as (\ref{structure1}). Since $ad_{e_1+\overline{e}_1}$ corresponds to the matrix
$$ \widetilde{A}:= \left[ \begin{array}{ll} \lambda & 0 \\ v & A \end{array} \right] , $$
we know that ${\mathfrak g}$ will be nilpotent if and only if $\lambda =0$ and $A$ is nilpotent (i.e., $A^{n-1}=0$). For any $i>1$, we have $\sum_{r=1}^n (C^r_{ri} + D^r_{ri}) = 0$, while
$$ \sum_{r=1}^n (C^r_{r1} + D^r_{r1}) = \lambda + \mbox{tr}(A) + \overline{\mbox{tr}(A)},$$
so ${\mathfrak g}$ is unimodular if and only if $\lambda + \mbox{tr}(A) + \overline{\mbox{tr}(A)} =0$. By (\ref{torsion}), the Chern torsion components are given by
\begin{equation} \label{torsion-aala}
 T^{\ast}_{ij}=0, \ \ \ T^1_{1i}=v_i, \ \ \ T^j_{1i}= A_{ij} + \overline{A_{ji} } , \ \ \ \ \ 2\leq i,j\leq n.
 \end{equation}
Thus the metric is K\"ahler if and only if $v=0$ and $A^{\ast} = -A$. As $\eta_a = \sum_r T^r_{ra}$, we have $\eta_i=v_i$ for $i>1$ and $\eta_1=- (\mbox{tr}(A) + \overline{\mbox{tr}(A)})$. So the metric $g$ is balanced if and only if $v=0$ and $\mbox{tr}(A) + \overline{\mbox{tr}(A)}=0$. For (v), the metric $g$ is pluriclosed if and only if equation (\ref{SKT}) holds for any $1\leq i<k\leq n$ and any $1\leq j< \ell \leq n$. Since $D^{\ast}_{\ast p}=T^{\ast}_{pq}=D^p_{11}=0$ for any $p,q>1$, one just needs to check the $j=1$ and $i=1$ case of (\ref{SKT}), which is
$$  -\sum_{r>1} T^r_{1k}\overline{C^r_{1\ell}} + \sum_{r>1} T^{\ell}_{1r}\overline{D^k_{r1}} - T^{\ell}_{k1}\overline{D^1_{11}} = \sum_{r>1} \{ T^r_{1k}D^{\ell}_{r1} + \overline{D^k_{r1}} T^{\ell}_{1r} \} + \lambda T^{\ell}_{1k} = 0, \ \ \ \forall \ k,\ell >1. $$
In other words, the metric $g$ is pluriclosed if and only if
\begin{equation} \label{Lafuente}
(A+A^{\ast})A + A^{\ast} (A+ A^{\ast}) + \lambda (A+A^{\ast}) = 0.
\end{equation}
This is discovered by Arroyo and Lafuente in \cite{AL} and they also showed that it is equivalent to:
\begin{equation} \label{Lafuente2} [A, A^{\ast }]=0, \ \ \ \mbox{and the real part of each eigenvalue of} \ A \ \mbox{is either} \ 0 \ \mbox{or} \ -\frac{\lambda}{2}.
\end{equation}
To see this equivalence, write $X=A+A^{\ast}$, then (\ref{Lafuente}) says that $[A, X] = \lambda X + X^2$. Since $X$ is Hermitian, one can take a unitary basis change of $\{ e_2, \ldots , e_n\}$ to make it into a real diagonal matrix. Group the same diagonal values into blocks, then the equation implies that $A$ is block-diagonal, and within each block, $X$ is a constant multiple (say $\lambda_i$) of the identity. By taking trace in this block, we know that the constant $\lambda_i$ must be either $0$ or $-\lambda$, and $[A,X]=0$. Thus $[A,A^{\ast}]=0$ which means $A$ is normal hence diagonalizable by unitary similarity. Note that $\lambda_i$ is twice the real part of an eigenvalue of $A$, so the assertion (\ref{Lafuente2}) holds. Thus (\ref{Lafuente}) implies (\ref{Lafuente2}). The converse is certainly true as well. This establishes the equivalence of (v).
\end{proof}

From the above, we see that unimodular plus balanced means that $\lambda =0$, $v=0$, while $A$ can be any complex matrix with trace in $i{\mathbb R}$. By Schur's triangulation theorem, $A$ is unitary similar to an upper triangular matrix. Analogously, unimodular plus pluriclosed means that the real part of only one eigenvalue of $A$ is $-\frac{\lambda}{2}$ while the others are zero, so after unitary similarity, $A$ is equal to a diagonal matrix with diagonal entries $\{ ib_2, \ldots , ib_{n\!-\!1}, ib_n\!-\!\frac{\lambda}{2}\}$, with all $b_j\in {\mathbb R}$, while $\lambda \in {\mathbb R}$ and $v\in {\mathbb C}^{n\!-\!1}$ are arbitrary.

In \cite{Paradiso}, Paradiso characterized all Hermitian almost abelian Lie algebra that are {\em locally conformally balanced}, and in \cite{DFFY}, Djebbar, Ferreira, Fino, and Youcef give characterization of all {\em locally conformally pluriclosed} ones.

Note that since we are primarily interested in compact examples, we need the Lie group $G$ to admit a co-compact lattice which  might be difficult to find for general solvable groups. But in the almost ableian case there is a nice sufficient condition given by Bock \cite[Prop 2.1]{Bock} (see also \S 4 of \cite{DFFY} and \cite{AO}, \cite{CM}).

Next let us look at the Chern curvature of an almost abelian Lie algebra. By (\ref{curvature}) and (\ref{CandD-aala}), we deduce that $R_{i\bar{\ast}\ast\bar{\ast}}=0$ and
\begin{equation*}
R_{1\bar{1}1\bar{1}} = -2\lambda^2-|v|^2, \ \ \ (R_{1\bar{1}i\bar{1}})=-A^{\ast}v, \ \ \ (R_{1\bar{1}i\bar{j}})=vv^{\ast} +[A, A^{\ast}] - \lambda (A + A^{\ast}),
\end{equation*}
for any $i,j>1$. By (\ref{shat}) and (\ref{CandD-aala}), we get
$$\hat{s}=-\lambda^2 -|\eta|^2 = -\lambda^2 - |v|^2 -(\mbox{tr}(A) + \overline{\mbox{tr}(A)})^2.$$
From these formula we conclude that

\begin{lemma} \label{lemma6}
Let ${\mathfrak g}$ be an almost abelian Lie algebra with a Hermitian structure $(J,g)$. Then it is Chern flat if and only if $\lambda=0$, $v=0$, and $[A, A^{\ast}]=0$. It has altered scalar curvature $\hat{s}=0$ if and only if $\lambda=0$, $v=0$ and $\mbox{tr}(A) + \overline{\mbox{tr}(A)}=0$. In the latter case ${\mathfrak g}$ is necessarily unimodular.
\end{lemma}

Using the definition of three Ricci curvatures, a direct calculation leads to the following:

\begin{lemma} \label{lemma7}
Assume that ${\mathfrak g}$ is as in Lemma \ref{lemma6} and is unimodular, namely, $\lambda + \mbox{tr}(A) + \overline{\mbox{tr}(A)}=0$. Then the following holds:
\begin{itemize}
\item $s=-\lambda^2 \leq 0, \ \mbox{and} \ s=0 \ \Longleftrightarrow \ Ric^{(1)}=0 \ \Longleftrightarrow \ \lambda =0$;
\item $\hat{s}=-2\lambda ^2-|v|^2 \leq 0, \ \mbox{and} \ \hat{s}=0 \ \Longleftrightarrow \ Ric^{(3)}=0 \ \Longleftrightarrow \ \lambda =0, \,v=0$;
\item $ Ric^{(2)}=0 \ \Longleftrightarrow \ R=0 \ \Longleftrightarrow \  \lambda =0, \,v=0,\, [A, A^{\ast}]=0$.
\end{itemize}
\end{lemma}

Next let us consider the {\em Chern K\"ahler-like} condition (cf. \cite{YZ16}, \cite{AOUV}), which means a Hermitian manifold whose Chern curvature tensor obeys the symmetry $R_{i\bar{j}k\bar{\ell}}= R_{k\bar{j}i\bar{\ell}}$. As is well-known, the first Bianchi identity implies that
\begin{equation*}
T^{\ell}_{ik,\bar{j}} = R_{k\bar{j}i\bar{\ell}} - R_{i\bar{j}k\bar{\ell}}
\end{equation*}
for any indices, where index after comma denotes covariant derivatives with respect to the Chern connection $\nabla$. So the Chern K\"ahler-like condition simply means that $\nabla_{\overline{X}}T=0$ for any type $(1,0)$ tangent vector $X$. Chern flat metrics are certainly Chern K\"ahler-like, and it would be very interesting to find examples of compact Chern K\"ahler-like manifolds that are not Chern flat. In \cite{ZZ-JGP}, it was shown that amongst all complex nilmanifolds, Chern K\"ahler-like ones are all Chern flat. For almost abelian ones, the phenomenon continues:

\begin{lemma} \label{lemma8}
Let ${\mathfrak g}$ be a unimodular almost abelian Lie algebra with a Hermitian structure $(J,g)$. Then $g$ is Chern K\"ahler-like if and only if it is Chern flat.
\end{lemma}

\begin{proof}
Assume that $g$ is Chern K\"ahler-like. Then we have $T^{\ell}_{ik,\bar{j}}=0$ for any $i,j,k,\ell$. By (\ref{Tderivativebar}), this means that
$$
 \sum_r \left(  T^j_{rk} \overline{\Gamma^i_{r\ell}} +  T^j_{ir} \overline{\Gamma^k_{r\ell}}  - T^r_{ik} \overline{\Gamma^r_{j\ell}}   \right)  =0.
$$
Take $i=\ell=1<k$, since $\Gamma =D$, while $T$ and $D$ are given by (\ref{torsion-aala}) and (\ref{CandD-aala}), we get from the $j=1$ case that $A^{\ast}v=0$, and from the $j>1$ case that
$$  \lambda (A+A^{\ast}) + [A^{\ast}, A] - v \,v^{\ast} =0. $$
Taking trace, we get $-\lambda^2-|v|^2=0$, hence $\lambda =0$, $v=0$, and $[A, A^{\ast}]=0$. Thus $R=0$.
\end{proof}

Next we mention the result by Fino and Paradiso \cite[Theorem 3.1]{FP2} which characterize {\em CYT} metrics on almost abelian Lie algebras.  Recall that a Hermitian metric is called {\em Calabi-Yau with torsion,} or {\em CYT} for short, if the first Ricci curvature of the Bismut connection vanishes. Denote by $\theta^b$, $\Theta^b$ the connection and curvature matrix of $\nabla^b$ under an admissible frame. Then one has
$$ \mbox{tr}(\Theta^b) = d (\mbox{tr}(\theta^b)) = d \left( \mbox{tr}(\theta) +\eta - \overline{\eta}\right) = d(\alpha - \overline{\alpha}),  $$
where $\alpha = (\lambda - \overline{\mbox{tr}(A)})\varphi_1 + \sum_{r=2}^n v_r\varphi_r$. By the structure equation (\ref{structure1}), and assuming that ${\mathfrak g}$ is unimodular, then one gets
\begin{equation*}
 \mbox{tr}(\Theta^b) =d(\alpha - \overline{\alpha}) = -(2|v|^2+3\lambda^2)\varphi_1\overline{\varphi}_1 + (\varphi_1+\overline{\varphi}_1)\sum_{i,j=2}^n \{ \overline{A_{ij}} v_i\varphi_j - A_{ij}  \overline{v}_i\overline{\varphi}_j \}.
\end{equation*}
Therefore it holds that
\begin{proposition}[\cite{FP2}] \label{prop8}
For any unimodular almost abelian Lie algebra ${\mathfrak g}$ with Hermitian structure $(J,g)$, it is CYT if and only if $\lambda =0$, $v=0$, and $\mbox{tr}(A) + \overline{\mbox{tr}(A)}=0$. In particular, it is CYT when and only when it is balanced.
\end{proposition}

Next let us prove Proposition \ref{prop1} stated in the introduction, which describes all Bismut torsion-parallel or Bismut K\"ahler-like metrics on almost abelian Lie algebras.

\begin{proof}[{\bf Proof of Proposition \ref{prop1}:}] Let $({\mathfrak g},J,g)$ be an almost abelian Lie algebra with Hermitian structure. Assume that $g$ is Bismut torsion-parallel ({\em BTP}), namely, $\nabla^bT=0$. The expression of the covariant derivatives of Chern torsion $T$ with respect to the Bismut connection $\nabla^b$  are given by formula (\ref{Tderivative}) and (\ref{Tderivativebar}) with $\Gamma =D$ replaced by $\Gamma^b = \Gamma + T= D+T$. So we know that {\em BTP} is equivalent to
\begin{eqnarray}
&&   \sum_r \left( - T^j_{rk} (D^r_{i\ell} + T^r_{i\ell}) -  T^j_{ir} (D^r_{k\ell} + T^r_{k\ell}) + T^r_{ik} (D^j_{r\ell} +T^j_{r\ell})   \right)    =0, \label{BTP1} \\
&& \sum_r \left(  T^j_{rk} (\overline{D^i_{r\ell}}+ \overline{T^i_{r\ell}})  +  T^j_{ir} ( \overline{D^k_{r\ell}} + \overline{T^k_{r\ell}})   - T^r_{ik} (\overline{D^r_{j\ell}}  + \overline{T^r_{j\ell}})  \right)  =0 \label{BTP2}
\end{eqnarray}
for any $1\leq i<k\leq n$ and any $j,\ell$. Now by (\ref{CandD-aala}), (\ref{torsion-aala}), and by taking $i=j=1<k,\ell$ in (\ref{BTP2}), we get
$$ \sum_{r>1} T^r_{1k} \overline{T^r_{1\ell} } =0,$$
which means that $(n-1)\times (n-1)$ matrix $X=(T^j_{1i})= A + A^{\ast}$ satisfies  $XX^{\ast}=0$, hence $X=0$. That is, the {\em BTP} assumption implies $A+A^{\ast} =0$ for almost abelian Lie algebras. With $A$ being skew-Hermitian, by (\ref{torsion-aala}) we know that the only possibly non-zero torsion components are $T^1_{1i}=v_i$. Letting $i=j=\ell =1<k$ in (\ref{BTP1}), we get $Av=0$.

Conversely, when $A+A^{\ast} =0$ and $Av=0$, the only possibly non-zero torsion components are $T^1_{1i}=v_i$, and it is easy to check that both (\ref{BTP1}) and (\ref{BTP2}) are satisfied. Thus for almost ableian Lie algebra, the {\em BTP} condition is equivalent to $A+A^{\ast}=0$ and $Av=0$.  This establishes the first statement in Proposition \ref{prop1}.

Now we check the {\em BKL} condition for almost abelian Lie algebra. By \cite{ZZ-pluriclosed}, we know that a Hermitian metric $g$ is {\em BKL} if and only if it is both {\em BTP} and pluriclosed. Now by Arroyo-Lafuente's result, (v) of Proposition \ref{propAL}, we know that an almost abelian Lie algebra will be pluriclosed when and only when $A$ is normal and whose eigenvalues have real part equal to either $0$ or $-\frac{\lambda}{2}$. But this condition is automatically satisfied when $A$ is skew-Hermitian. So for almost abelian Lie algebras, the {\em BKL} condition and the {\em BTP} conditions coincide. Note that the Lie algebra will be unimodular in this case if and only if $\lambda =0$. This completes the proof of Proposition \ref{prop1}.
\end{proof}

On general Hermitian manifolds, {\em BKL} and {\em BTP} metrics were studied by several authors in recent years, for more discussions on such special Hermitian structures, we refer the readers to  \cite{AOUV},  \cite{AP},  \cite{FT}, \cite{FTV}, \cite{LS}, \cite{YZ16}, \cite{YZZ},  \cite{ZZ-pluriclosed},  \cite{ZZ-JGP}, \cite{ZZ-JGA},  \cite{ZZ-BTP}, \cite{Zheng} and the references therein.

Astheno-K\"ahler metrics were introduced by Jost-Yau in \cite{JY}, as a relaxation to the K\"ahler condition. It is defined by $\partial \overline{\partial} (\omega^{n-2})=0$ where $\omega$ is the K\"ahler form and $n$ is the complex dimension of the manifold. This special type of Hermitian metrics has seen a number of recent development, and we refer the readers to \cite{CR}, \cite{Correa}, \cite{FGV1}, \cite{FGV}, \cite{FT1}, \cite{LU}, \cite{Popovici}, \cite{ST} for instance.

\begin{proof}[{\bf Proof of Proposition \ref{prop2}:}]  Let $({\mathfrak g}, J,g)$ be a an almost abelian Lie algebra with Hermitian structure. Then there exists a unitary frame $e$ so that the structure equation takes the form of (\ref{structure1}), namely, $d\varphi_1=-\lambda \varphi_1\overline{\varphi}_1$, $\, d\varphi_i= -\overline{v}_i\varphi_1\overline{\varphi}_1 + \sum_j \overline{A_{ij}} (\varphi_1+\overline{\varphi}_1)\varphi_j$. Here and below we set the index range to be  $2\leq i,j, k,\ell \leq n$. The K\"ahler form $\omega$ is given by $\omega = \sqrt{-1} \left( \varphi_1\overline{\varphi}_1 + \sum_i \varphi_i\overline{\varphi}_i \right)$. Let us write $H=A+A^{\ast}$. Since $H$ is Hermitian, by a unitary change of $e$  we may assume that $H$ is diagonal, say $H=\mbox{diag}\{ h_2, \ldots , h_n\}$. We have
\begin{eqnarray*}
-\sqrt{-1} \partial \omega & = &  - \varphi_1\overline{\varphi}_1 \sum_i v_i\varphi_i + \varphi_1 \sum_{i,j} H_{ij} \varphi_i\overline{\varphi}_j  \, = \, - \varphi_1\overline{\varphi}_1 \sum_i v_i\varphi_i +  \varphi_1 \sum_i h_i \varphi_i\overline{\varphi}_i, \\
\partial \omega  \wedge \overline{\partial} \omega & = & - \varphi_1\overline{\varphi}_1 \big(\sum_{i} h_i \varphi_i\overline{\varphi}_i\big)^2 \, = \, - \varphi_1\overline{\varphi}_1 \sum_{i,j} h_ih_j \,\varphi_i\overline{\varphi}_i \varphi_j\overline{\varphi}_j \,,  \\
\sqrt{-1} \partial  \overline{\partial} \omega  & = & - \varphi_1\overline{\varphi}_1 \sum_{i,j} L_{ij} \varphi_i\overline{\varphi}_j, \ \
\end{eqnarray*}
where $L=\lambda H + A^{\ast}H + HA = \lambda H + H^2 + [A^{\ast}, A]$. Now assume that $n\geq 4$. The astheno-K\"ahler condition means that
\begin{equation} \label{eq:astheno}
 \frac{1}{n-2} \partial  \overline{\partial} \omega^{n-2} = \partial  \overline{\partial} \omega \wedge \omega^{n-3} + (n-3) \partial \omega  \wedge \overline{\partial} \omega \wedge \omega^{n-4} = 0.
 \end{equation}
For any $2\leq i<j\leq n$, if we wedge the above equality with $\varphi_i \overline{\varphi}_j$, then the second term on the right hand side produce zero, while the first term gives us $L_{ij}=0$, $\forall \,i\neq j$. So $L$, hence $[A^{\ast}, A]$, is diagonal. On the other hand, the  diagonal entries
$$ [A^{\ast}, A]_{ii} = [H, A]_{ii} = h_iA_{ii} - A_{ii}h_i = 0, $$
so we conclude that $[A^{\ast}, A]=0$. Thus $L=\lambda H + H^2$. Now if we fix  any $i$ and wedge the equality (\ref{eq:astheno}) with $\varphi_i \overline{\varphi}_i$, we get
\begin{eqnarray*}
0 & = & \sum_{j\neq i} L_{jj} + \sum_{j\neq k; \,j,k\neq i} h_jh_k \\
& = & \mbox{tr}(L) - L_{ii} + (\sum_{j\neq i} h_j)^2 - \sum_{j\neq i} h_j^2\\
& = & \mbox{tr}(L) - L_{ii} + (h-h_i)^2 - \sum_j h_j^2 + h_i^2 \\
& = & \lambda h + \mbox{tr}(H^2) - \lambda h_i - h_i^2  +h^2-2hh_i + h_i^2 - \mbox{tr}(H^2) +h_i^2 \\
& = & (\lambda h+h^2) -(\lambda +2h)h_i + h_i^2 \\
& = & ( \lambda + h-h_i) (h- h_i),
\end{eqnarray*}
where $h$ stands for  $\mbox{tr}(H)$.  Therefore each $h_i$ is equal to either $h$ or $\lambda +h$. Suppose there are $k$ entries amongst $h_2$ through $h_n$ that equals to $h$, while the other $(n-1-k)$ entries equal to $\lambda +h$, then we have $h = h_2+ \cdots + h_n = kh + (n-1-k)(\lambda +h)$, therefore,
\begin{equation} \label{eq:kandh}
 (n-1-k)\lambda + (n-2)h = 0.
 \end{equation}
To summarize, if $(J,g)$ is a Hermitian structure on an almost abelian  Lie algebra ${\mathfrak g}$ and $g$ is astheno-K\"ahler, then $[A^{\ast}, A]=0$, so $A$ is unitary similar to a diagonal matrix $\mbox{diag}\{ \lambda_2, \ldots , \lambda_n\}$ where $\lambda_i$ are eigenvalues of $A$. Furthermore, twice of the real part of $\lambda_i$ is equal to either $\lambda +h$ or $h$, where $h$ is the sum of $2\mbox{Re}(\lambda_i)$ for all $2\leq i\leq n$. If we denote by $k$ the number of these eigenvalues with $2\mbox{Re}(\lambda_i)=h$, then $k$ and $h$ obeys the equation (\ref{eq:kandh}). Equivalently, $A$ can be described in the following way: there exists a unitary $(n-1)\times (n-1)$ matrix $U$ such that
$$ U A U^{\ast} = \left[ \begin{array}{ccc} \lambda_2 & & \\ & \ddots & \\ && \lambda_n \end{array} \right], \ \ \ \ \ \left[ \begin{array}{ccc} 2\mbox{Re}(\lambda_2) & & \\ & \ddots & \\ && 2\mbox{Re}(\lambda_n) \end{array} \right] = \left[ \begin{array}{cc} hI_k & \\  &(\lambda +h)I_{n-1-k} \end{array} \right] $$
where $k$ is an integer in the range $0\leq k\leq n-1$ and $h$ is determined by $k$ and $\lambda$ via (\ref{eq:kandh}).

Note that $h=\mbox{tr}(A)+ \overline{\mbox{tr}(A)}$, so ${\mathfrak g}$ is unimodular if and only if $\lambda +h =0$. In this case (\ref{eq:kandh}) becomes $(1-k)\lambda =0$, thus for the real part of eigenvalues of $A$, one of them equal to $-\frac{\lambda}{2}$ while the other $n-2$ of them are equal to $0$. By part (ii) and (v) of Proposition \ref{propAL}, we know that this is exactly the condition for $g$ to be pluriclosed. This completes the proof of Proposition \ref{prop2}.   \end{proof}

\vspace{0.3cm}

\section{Lie algebras with $J$-invariant abelian ideal of codimension $2$}

Now let us consider Lie algebras with codimension $2$ abelian ideals. Let ${\mathfrak g}$ be a Lie algebra and ${\mathfrak a}\subseteq {\mathfrak g}$ be an abelian ideal of codimension $2$. Note that such ${\mathfrak g}$ will always be solvable of step at most $3$, but in general it will not be of $2$-step solvable.

Suppose $(J,g)$ is a Hermitian structure on ${\mathfrak g}$. Then the codimension of ${\mathfrak a}_J:= {\mathfrak a}\cap {\mathfrak a}$ is either $2$ or $4$, and it is $2$ if and only if $J{\mathfrak a}={\mathfrak a}$. In this article we will focus on this case, and leave the other case to a future project.

Analogous to the almost ableian case, we may take a unitary basis $\{ e_1, \ldots , e_n\}$ of ${\mathfrak g}^{1,0}$ so that
$$ {\mathfrak a} = \mbox{span}_{\mathbb R} \{ e_i+\overline{e}_i, \, \sqrt{-1}(e_i-\overline{e}_i); \ 2\leq i\leq n\} . $$
We will again call such a basis an {\em admissible frame}. Since ${\mathfrak a}$ is abelian and is an ideal, we have
\begin{equation*}
C^{\ast}_{ij}=D^j_{\ast i} =C^1_{\ast \ast} = D^{i}_{1\ast} =D^{\ast}_{1i} = 0, \ \ \ \ \ \forall \ 2\leq i,j\leq n.
\end{equation*}
So the only possibly non-zero components of $C$ and $D$ are
\begin{equation}  \label{CD-XYZ}
C^{j}_{1i}= X_{ij}, \ \ \ D^1_{11}=\lambda, \ \ \ D^j_{i1}=Y_{ij}, \ \ \ D^1_{ij}=Z_{ij}, \ \ \ D^1_{i1}=v_i, \ \ \ \ \ 2\leq i,j\leq n,
\end{equation}
where $\lambda \geq 0$, $v\in {\mathbb C}^{n-1}$ is a column vector and $X$, $Y$, $Z$ are $(n-1)\times (n-1)$ complex matrices. Here and from now on we have rotated the angle of $e_1$ to assume that $D^1_{11}\geq 0$. Note that when $\{ e_2, \ldots , e_n\}$ is changed to $\{ \tilde{e}_2, \ldots , \tilde{e}_n\}$  by a unitary matrix $U$, then $v$ is changed to $Uv$, $X$ and $Y$ are changed to $UXU^{\ast}$ and $UYU^{\ast}$ respectively, while $Z$ is changed to $UZ\,^t\!U$. For convenience, let us write $\varphi$ for the column vector $^t\!(\varphi_2, \ldots , \varphi_n)$, then the structure equation (\ref{structure}) becomes (\ref{structure2}) stated in \S 1:
\begin{equation*}
\left\{ \begin{split} d\varphi_1 \, = \, -\lambda \varphi_1\overline{\varphi}_1 , \hspace{4.1cm} \\
d\varphi \, = \, - \varphi_1\overline{\varphi}_1 \overline{v} -\varphi_1\,^t\!X \varphi + \overline{\varphi}_1 \overline{Y} \varphi - \varphi_1 \overline{Z} \,\overline{\varphi}. \end{split} \right.
\end{equation*}
By the Bianchi equation (\ref{CC}) - (\ref{CDbar}), or equivalently, by $d^2\varphi =0$, we get the system of matrix equations (\ref{XYZ}) stated in \S 1:
\begin{equation*}
\left\{ \begin{split} \lambda (X^{\ast}\!+Y)+ [X^{\ast} ,Y]  -  Z\overline{Z} \, = \, 0, \\
\lambda Z - ( Z \,^t\!X + Y Z )\, =\, 0. \hspace{1.2cm} \end{split} \right.
\end{equation*}
By (\ref{torsion}), we know that the possibly non-zero components of the Chern torsion are:
\begin{equation} \label{torsion-XYZ}
T^1_{1i}=v_i, \ \ \ \ T^1_{ij}=Z_{ji}-Z_{ij}, \ \ \ \ T^j_{1i}=Y_{ij}-X_{ij}, \ \ \ \ \ 2\leq i,j\leq n.
\end{equation}

\begin{proof} [{\bf Proof of Proposition \ref{prop3}:}]
Let ${\mathfrak g}$ be a Lie algebra with an abelian ideal ${\mathfrak a}$ of codimension $2$ and with a Hermitian structure $(J,g)$, and assume that $J{\mathfrak a}={\mathfrak a}$. By our discussion above, there exists a unitary frame under which the structure equation is given by (\ref{structure2}), where $\lambda \geq 0$, $v\in {\mathbb C}^{n-1}$ is a column vector and $X$, $Y$, $Z$ are complex $(n-1)\times (n-1)$ matrices satisfying (\ref{XYZ}). By (\ref{unimo}), we know that ${\mathfrak g}$ is unimodular if and only if $\sum_{r=1}^n (C^r_{ri}+D^r_{ri})=0$ for each $1\leq i\leq n$. By (\ref{CD-XYZ}), we see that this number is $0$ when $i>1$ and is $\lambda +\mbox{tr}(Y)- \mbox{tr}(X)$ when $i=1$. So (i) holds. Similarly, Gauduchon's torsion $1$-form $\eta$ has components $\eta_i = \sum_{r=1}^nT^r_{ri}$.  By (\ref{torsion-XYZ}), we get $\eta_1=\mbox{tr}(X)- \mbox{tr}(Y)$ and $\eta_i=v_i$ for $i>1$, so (ii) holds. Also by (\ref{torsion-XYZ}), we know that $T=0$ if and only if $v=0$ plus $X=Y$ and $\,^t\!Z=Z$. Now assuming that ${\mathfrak g}$ is unimodular. Then by (i) we get $\lambda =0$, hence the first equation of (\ref{XYZ}) reads $[X^{\ast},X]= Z\overline{Z}=ZZ^{\ast}$. Taking trace on both sides, we get $\mbox{tr}(ZZ^{\ast})=0$ thus $Z=0$ and $X$ is normal. So by a unitary change of $\{ e_2, \ldots , e_n\}$ if necessary, we may assume that $X$ is diagonal. This leads to the structure equation stated in case (iii) of the proposition.

For part (iv), by Lemma \ref{lemma2}  we know that $g$ will be pluriclosed if and only if equation (\ref{SKT}) is satisfied. When $i$ or $j$ is $1$, the equation (\ref{SKT}) automatically holds. Take $i=j=1<k,\ell$ in (\ref{SKT}), by (\ref{CD-XYZ}) and (\ref{torsion-XYZ}) we get
$$ \sum_{r>1} \left( (X_{kr}-Y_{kr})\overline{X_{\ell r} } + (Z_{rk}-Z_{kr})\overline{Z_{r\ell} } + (Y_{r\ell} -X_{r\ell}) \overline{Y_{rk}} +\lambda ( Y_{k\ell} -X_{k\ell}) \right) =0. $$
That is,
\begin{equation*}
 XX^{\ast}-YX^{\ast} + \,^t\!Z \overline{Z} - Z \overline{Z}    + Y^{\ast} Y - Y^{\ast} X + \lambda (Y-X) = 0,
 \end{equation*}
namely (\ref{SKT-XYZ}) holds. So $g$ is pluriclosed if and only if equations (\ref{XYZ}) and (\ref{SKT-XYZ}) hold. This proves  part (iv) thus we have  completed the proof of Proposition \ref{prop3}.
\end{proof}

Note that if we take the trace of (\ref{SKT-XYZ}), we would get
$$ |\!|X-Y|\!|^2 + |\!|Z|\!|^2 = \mbox{tr}(Z\overline{Z})+ \lambda ( \mbox{tr}(X) -\mbox{tr}(Y)) = \lambda ( \mbox{tr}(X) + \overline{ \mbox{tr}(X) } ), $$
where the last equality is from taking trace on the first equation of (\ref{XYZ}). When $\lambda =0$ or when $\mbox{tr}(X) + \overline{ \mbox{tr}(X) } =0$, we would have $Z=0$ and $X=Y$. If ${\mathfrak g}$ is unimodular, then $\lambda=0$ and $[X,X^{\ast}]=0$, so $X$ is unitary similar to a diagonal matrix. That is, when ${\mathfrak g}$ is unimodular and $\lambda =0$, pluriclosed metrics admit unitary frame under which the structure equation becomes:
\begin{equation*}
 d\varphi_1=0, \ \ \ d\varphi_i = -\overline{v}_i\varphi_1\overline{\varphi}_1 + (\overline{a}_i\overline{\varphi}_1-a_i\varphi_1)\varphi_i, \ \ \ \ i>1.
\end{equation*}
Note that when $v\neq 0$, the above metric is neither balanced nor Chern flat. Also, there are plenty of solutions to equations (\ref{XYZ}) and (\ref{SKT-XYZ}) with $\lambda >0$. For instance, for any positive constant $\lambda >0$, if we let
$$ Z=0, \ \ \ \ X = -Y = \frac{\lambda }{2}\,\mbox{diag} \{ 1, 0, \ldots , 0\}, $$
and let $v$ be arbitrary, then ${\mathfrak g}$ would be unimodular and  equations (\ref{XYZ}) and (\ref{SKT-XYZ}) are satisfied, hence $g$ is pluriclosed.

\begin{proof}[{\bf Proof of Proposition \ref{prop4}:}]
Let $({\mathfrak g}, J,g)$ be a Hermitian structure on a Lie algebra which contains a $J$-invariant abelian ideal of codimension $2$. First let us examine the  behavior of the Chern curvature tensor. By (\ref{curvature}), (\ref{CD-XYZ}),  and (\ref{Ric1}) - (\ref{Ric3}), we have
\begin{eqnarray}
Ric^{(1)} & = &  \left[ \begin{array}{cc} s,  & 0 \\ 0, & 0 \end{array} \right], \ \ \ \ \ \ \mbox{where} \ \ \ s=-\lambda (2\lambda +\mbox{tr}(Y) +\overline{\mbox{tr}(Y)} ), \nonumber \\
Ric^{(2)} & = &  \left[ \begin{array}{cc} -(|v|^2+|\!|Z|\!|^2+2\lambda^2),  & -(\,^t\!v Z^{\ast}+v^{\ast}Y) \\ - (Z\bar{v} + Y^{\ast}v), & vv^{\ast} +ZZ^{\ast} + [Y,Y^{\ast}]-\lambda (Y+Y^{\ast}) \end{array} \right],   \label{Ric2-XYZ}\\
Ric^{(3)} & = &  \left[ \begin{array}{cc} \hat{s},  & -\,^t\!v \overline{Z} \\ - Y^{\ast}v, & 0 \end{array} \right], \ \ \ \ \  \mbox{where} \ \ \ \hat{s}=-|v|^2 -\lambda (2\lambda +\mbox{tr}(Y) +\overline{\mbox{tr}(X)} ).  \nonumber
\end{eqnarray}
In the last line we used the fact that $\mbox{tr}(Z\overline{Z})=\lambda (\mbox{tr}(Y) + \overline{\mbox{tr}(X) } ) $ by taking trace on the first equation of (\ref{XYZ}).
Now if $R=0$, then by $Ric^{(2)}_{1\bar{1}}=0$, we get $\lambda =0$, $v=0$ and $Z=0$. Thus the only possibly non-zero components of $D$ are $D^j_{i1}=Y_{ij}$ for $i,j>1$. By (\ref{curvature}), for $2\leq k,\ell \leq n$, we have
$$ \left( R_{1\bar{1}k\bar{\ell}}\right) = YY^{\ast}-Y^{\ast}Y. $$
So $Y$ is normal and can be diagonalizable via unitary similarity. On the other hand, when $\lambda =0$ and $Z=0$, the first equation of (\ref{XYZ}) says that $[X^{\ast},Y]=0$, so if we write
$$ Y = \left[ \begin{array}{ccc} \lambda_1 I_{n_1} & & \\ & \ddots & \\ && \lambda_rI_{n_r} \end{array} \right], $$
where $n_1+ \cdots +n_r=n-1$ and $\lambda_1, \ldots , \lambda_r$ are distinct, then $X$ would be block-diagonal
$$ X = \left[ \begin{array}{ccc} X_1 & & \\ & \ddots & \\ && X_r \end{array} \right] $$
where $X_i$ is any $n_i\times n_i$ complex matrix, $1\leq i\leq r$. Therefore if $R=0$, then $\lambda =0$, $v=0$, $Z=0$, $Y$ is normal, and $[X^{\ast}, Y]=0$. By a unitary change of the frame, $Y$ can be made diagonal while $X$ is block-diagonal as above. The converse is certainly true, namely, in this case the only possibly non-zero components of $D$ are $D^i_{i1}=Y_{ii}$, and by (\ref{curvature}) one easily checks that $R=0$. This establishes (i).   Part (ii) is immediate from the above formula for the first Chern Ricci $Ric^{(1)}$ and Chern scalar curvature $s$. For part (iii), assume that the second Chern Ricci is non-neagtive. By looking at the component $Ric^{(2)}_{1\bar{1}}$ in  (\ref{Ric2-XYZ}), we know that we must have $\lambda=0$, $v=0$, $Z=0$. The lower right corner of (\ref{Ric2-XYZ}) now says that $[Y, Y^{\ast}]\geq 0$. But the matrix has zero trace, so it must be identically zero, and we end up in the  $R=0$ situation.

For part (iv), note that Gauduchon's torsion $1$-form $\eta$ has components $\eta_1=\mbox{tr}(Y-X)$ and $\eta_i=v_i$ for $i>1$. So by (\ref{CD-XYZ}), (\ref{structure2}) and (\ref{torsion-XYZ}) we obtain the following
 \begin{eqnarray*}
 d\eta & = &  d(\sum_r \eta_r\varphi_r) \, = \, \eta_1d\varphi_1 - \sum_{r>1} v_r \left( \overline{v}_r \varphi_1\overline{\varphi}_1 +\sum_{i>1} \{ X_{ir} \varphi_1\varphi_i +\overline{Y}_{ri}\varphi_i\overline{\varphi}_1  + \overline{Z}_{ri}\varphi_1\overline{\varphi}_i \}  \right)  \\
 & = & -(\lambda \eta_1+|v|^2) \varphi_1\overline{\varphi}_1 - \sum_{i,r>1} v_r \left(  X_{ir} \varphi_1\varphi_i +\overline{Y}_{ri}\varphi_i\overline{\varphi}_1  + \overline{Z}_{ri}\varphi_1\overline{\varphi}_i   \right)
\end{eqnarray*}
Since the trace of the connection matrices for Chern and Bismut are related by $\mbox{tr}(\theta^b)=\mbox{tr}(\theta)+\eta -\overline{\eta}$, so the Bismut Ricci form is given by
\begin{eqnarray*}
\mbox{tr}(\Theta^b) & = & \mbox{tr}(\Theta) + d\eta -d\overline{\eta}  \ \,= \ \,s \varphi_1\overline{\varphi}_1 + d\eta -d\overline{\eta} \\
& = & -\left( 2|v|^2 + \lambda \{ 2\lambda + 2\mbox{tr}(Y) + 2\overline{\mbox{tr}(Y)} -\mbox{tr}(X) -\overline{\mbox{tr}(X)} \} \right)\varphi_1\overline{\varphi}_1 + \\
& & + \sum_{i,r>1} \left( \overline{v}_r\overline{X_{ir}} \overline{\varphi}_1\overline{\varphi}_i - v_rX_{ir} \varphi_1\varphi_i \right) - \sum_{i,r>1} \left( v_r\overline{Z_{ri}} +\overline{v}_r Y_{ri} \right)
\varphi_1\overline{\varphi}_i \\
& & - \sum_{i,r>1} \left( v_r\overline{Y_{ri}} +\overline{v}_r Z_{ri} \right)
\varphi_i\overline{\varphi}_1
\end{eqnarray*}
That is,
\begin{equation} \label{BismutRicci}
\left\{ \begin{split}  (Ric^b)^{2,0} \, = \, \frac{1}{2} \left[ \begin{array}{cc} 0 & ^t\!v \,^t\!X \\ - Xv & 0 \end{array} \right] ,   \hspace{2.6cm} \\
 (Ric^b)^{1,1} \, =  \, \left[ \begin{array}{cc} s^b & - (v^{\ast} Y + \,^t\!v \overline{Z}) \\ - (Y^{\ast}v +\,^t\!Z\overline{v}) & 0 \end{array} \right] , \end{split} \right.
\end{equation}
where the {\em Bismut scalar curvature} $s^b$ is given by (\ref{Bismutscalar}) as follows:
\begin{equation*}
s^b \, = \, -2|v|^2 - \lambda (2\lambda + 2\mbox{tr}(Y) + 2\overline{\mbox{tr}(Y)} -\mbox{tr}(X) -\overline{\mbox{tr}(X)}).
\end{equation*}
When ${\mathfrak g}$ is unimodular, $\lambda = \mbox{tr}(X) - \mbox{tr}(Y)$, and $s^b=-2|v|^2 -\lambda (\mbox{tr}(Y) + \overline{\mbox{tr}(Y)})$. By (\ref{BismutRicci}), we know that the rank of the $(2,0)$-part of $Ric^b$ is either $2$ or $0$ (when $Xv=0$), while the $(1,1)$-part of $Ric^b$ has rank at most $3$, or at most $1$ when $Y^{\ast}v+ \, ^t\!Z\overline{v}=0$. The entire  $Ric^b=0$ if and only if $s^b=0$ and $Xv=0$, $Y^{\ast}v+ \, ^t\!Z\overline{v}=0$. This proves part (iv) and completes the proof of Proposition \ref{prop4}.
\end{proof}

The remaining part of the article is devoted to the proof of Proposition \ref{prop5}, where we want to characterize Bismut torsion-parallel ({\em BTP})  and Bismut K\"ahler-like ({\em BKL}) metrics on Lie algebras with $J$-invariant abelian ideal of codimension $2$. Since {\em BKL} means {\em BTP} plus pluriclosed \cite{ZZ-pluriclosed}, we will just focus on the {\em BTP} condition.

\begin{proof}[{\bf Proof of Proposition \ref{prop5}:}]
Again let $({\mathfrak g},J,g)$ be a Hermitian structure on a unimodular Lie algebra with a $J$-invariant abelian ideal of codimension $2$. Assume that $g$ is  {\em BTP}. This means that both (\ref{BTP1}) and (\ref{BTP2}) are satisfied, which now takes the form
\begin{equation} \label{BTP-XYZ}
\left\{ \begin{split}   \sum_r \left(  T^j_{rk} \tilde{\Gamma}^r_{i\ell}  +  T^j_{ir} \tilde{\Gamma}^r_{k\ell}  - T^r_{ik} \tilde{\Gamma}^j_{r\ell}   \right)    =0, \\
 \sum_r \left(  T^j_{rk} \overline{ \tilde{\Gamma}^i_{r\ell} }   +  T^j_{ir}  \overline{ \tilde{\Gamma}^k_{r\ell} }    - T^r_{ik} \overline{\tilde{\Gamma}^r_{j\ell}}    \right)  =0
\end{split} \right.
\end{equation}
for any indices $1\leq i<k\leq n$ and any $1\leq j, \,\ell \leq n$, where $\tilde{\Gamma}=D+T$ are the connection coefficients for the Bismut connection. For simplicity, let us denote by
\begin{equation*}
A=\,^t\!Z-Z, \ \ \ \ B=Y-X.
\end{equation*}
Then the only possibly non-zero components of $T$ and $\tilde{\Gamma}$ are
\begin{equation*}
\left\{ \begin{split}  T^1_{1i}=v_i, \ \ \ T^1_{ij}=A_{ij}, \ \ \ T^j_{1i}=B_{ij}; \hspace{3.7cm} \\
\tilde{\Gamma}^1_{11}=\lambda , \ \ \
\tilde{\Gamma}^1_{1i}=v_i, \ \ \ \tilde{\Gamma}^1_{ij}=Z_{ji}, \ \ \ \tilde{\Gamma}^j_{i1}=X_{ij}, \ \ \ \tilde{\Gamma}^j_{1i}=B_{ij},
\end{split} \right.
\end{equation*}
where $2\leq i,j \leq n$. We remark that when the basis $\{ e_2,\ldots , e_n\}$ is changed to $\{ \tilde{e}_2, \ldots , \tilde{e}_n\}$ by a unitary matrix $U\in U(n-1)$, the matrix $A$ is changed into $UA\,^t\!U$, just like $Z$, while $B$ is changed to $UBU^{\ast}$, just like $X$ or $Y$. So strictly speaking we should denote the entries of $B$ as $B_i^{\ j}$ or $B_{i\bar{j}}$, but for the sake of simplicity we will just write it as $B_{ij}$.

Note that for both $T$ and $\tilde{\Gamma}$, the components are zero unless at least one of the indices is $1$. If we take $i,j,k,\ell >1$ in (\ref{BTP-XYZ}), then only the $r=1$ terms survive, and we get
\begin{eqnarray}
 B_{kj} Z_{\ell i} - B_{ij} Z_{\ell k} - A_{ik} B_{\ell j} & = &  0  ,   \label{eq1} \\
 B_{kj}  \overline{ B_{\ell i} } - B_{ij} \overline{  B_{\ell k} } - A_{ik} \overline{Z_{\ell j} } & = & 0 ,  \label{eq2}
\end{eqnarray}
where $2\leq i,j,k, \ell \leq n$. Write $t_B=\mbox{tr}(B)$ and $t_Z=\mbox{tr}(Z)$.  By letting $k=j$ and summing up, or $k=\ell$ and summing up, or $j=\ell$ and summing up, we get
\begin{eqnarray}
&& \left\{ \begin{split}  t_B Z - Z\,^t\!B + BA = 0 \\ t_B \overline{B} - \overline{B}\,^t\!B +  \overline{Z}A = 0  \end{split} \right.    \label{eq3} \\
&& \left\{ \begin{split} \,^t\!Z B -t_Z B -  AB = 0 \\  B^{\ast}B -\overline{t_B} B -  A\overline{Z} = 0  \end{split} \right.   \label{eq4} \\
&& \left\{ \begin{split} BZ - \,^t\!Z\,^t\!B + t_B A  = 0  \\ B \overline{B} - B^{\ast} \,^t\!B + \overline{t_Z} A = 0 \end{split} \right. \label{eq5}
\end{eqnarray}
Next we consider the cases where at least one of the four indices is $1$. If $i=j=\ell =1$, then by the fact that  $\tilde{\Gamma}^{i}_{11} = \tilde{\Gamma}^{1}_{i 1}= 0$ for any $i>1$, (\ref{BTP-XYZ}) yields
\begin{equation} \label{Xv}
Xv = X^{\ast} v=0.
\end{equation}
Similarly, when exactly one of the indices is $1$, the $\ell =1$ case of (\ref{BTP-XYZ}) automatically holds, while the $j=1$ and  $i=1$ cases lead to
\begin{eqnarray}
&& \left\{ \begin{split}  v_k Z_{\ell i} - v_i Z_{\ell k} - v_{\ell} A_{ik} =0, \\
v_k \overline{B_{\ell i}} - v_i \overline{B_{\ell k}} -  \overline{v_{\ell} }A_{ik} = 0,
 \end{split} \right.   \ \ \ \ \forall \ 2\leq i,k,\ell \leq n; \label{vZBA} \\
&& \left\{ \begin{split} v_{\ell} B_{kj} - v_k B_{\ell j} = 0, \\  \overline{v_{\ell} }B_{kj} -  v_k\overline{Z_{\ell j}} = 0,  \end{split} \right.  \ \ \ \ \ \ \ \ \ \ \ \ \forall \ 2\leq j,k,\ell \leq n.  \label{vBZ}
\end{eqnarray}
Now consider the cases when exactly two of the four indices are $1$,  (\ref{BTP-XYZ})  yields the following
\begin{eqnarray}
&& BA=Z\,^t\!B, \ \ \ \ \overline{Z}A = \overline{B} \,^t\!B; \label{BA} \\
 && [B, X] = [B, X^{\ast} ] =\lambda B; \label{XB}\\
&& XA+A\,^t\!X = X^{\ast}A + A\overline{X} =\lambda A. \label{XA}
\end{eqnarray}
Of course the Lie algebra ${\mathfrak g}$ will be unimodular if and only if
\begin{equation} \label{unimodular}
\lambda + t_B = 0.
\end{equation}
The metric $g$ will be {\em BTP} when and only when (\ref{XYZ}) and (\ref{eq1}) -- (\ref{unimodular}) are all satisfied. When $v\neq 0$, it is easy to find out all the solutions to this system of matrix equations. By a unitary change of $\{ e_2, \ldots , e_n\}$ if necessary, we may assume that $v_2>0$ and $v_3=\cdots = v_n=0$. By (\ref{vZBA}) and (\ref{vBZ}), we know all rows  except the first of $Z$ are zero, and $B=\overline{Z}$. The other equations now imply that
$$ \lambda =0, \ \  X = \left[ \begin{array}{cc} 0 & 0 \\ 0 & X_1 \end{array}  \right] , \ \ \ Z = \left[ \begin{array}{cc} 0 & \,^t\!z \\ 0 & 0 \end{array}  \right], $$
and $X_1z = X_1^{\ast}z=0$, $[X_1, X_1^{\ast}]=0$. So depending on whether $z=0$ or not, we end up in one of the two cases for {\em BTP} metrics with $v\neq 0$ on unimodular ${\mathfrak g}$: there exists unitary frame so that the structure equation takes the form:
\begin{eqnarray*}
&& \left\{ \begin{split} d\varphi_1 = 0,  \hspace{4.2cm}\\
d\varphi_2 = - v_2 \varphi_1\overline{\varphi}_1, \hspace{3cm}\\
d\varphi_i = (\overline{a}_i\overline{\varphi}_1 - a_i \varphi_1)\varphi_i, \ \ \ 3\leq i\leq n;
\end{split} \right.   \ \ \ \ \ \ \ \ \ \ \ \ \ \ \ \ \ \ \ \          \\
  \mbox{or} && \left\{ \begin{split} d\varphi_1 =d\varphi_3 =  0, \hspace{3.15cm} \\
d\varphi_2 = - v_2 \varphi_1\overline{\varphi}_1 + p \,(\overline{\varphi}_1\varphi_3 - \varphi_1 \overline{\varphi}_3),  \hspace{0.1cm} \\
d\varphi_i = (\overline{a}_i\overline{\varphi}_1 - a_i \varphi_1)\varphi_i, \ \ \ 4\leq i\leq n.
\end{split} \right.
\end{eqnarray*}
They are respectively (\ref{BTPv1}) and (\ref{BTPv2}) stated in \S 1 before.
Here $v_2>0$, $p>0$, and $a_i$ are arbitrary complex numbers.

Next we consider the $v=0$ case. First we claim that $\lambda $ must zero. Assume otherwise, then by comparing (\ref{eq3}) with (\ref{BA}), we see that $B$ and $Z$ must be zero, which forces $\lambda = -t_B =0$ since ${\mathfrak g}$ is unimodular. So $\lambda $ must be zero. Next we claim that $BZ=0$. Since $t_B=0$, by the first line of (\ref{eq5}), we know that $BZ$ is symmetric. On the other hand, by the first equation of (\ref{BA}), we get
$$ \,^t(BA) = B\,^t\!Z = B(A+Z) = BA + BZ. $$
Hence $BZ$ is skew-symmetric. Therefore, we must have $BZ=0$ and the claim is proved. So for unimodular ${\mathfrak g}$ with {\em BTP} metric and with $v=0$, the equations (\ref{XYZ}) and (\ref{eq1}) -- (\ref{unimodular}) become
\begin{equation} \label{BTP-final}
\left\{   \begin{split} \lambda = t_B=0, \ \ v=0,  \ \ BZ=0, \ \
XA, \ X^{\ast}A, \ BA=B\,^t\!Z \ \  \mbox{are symmetric}, \\
[B,X] = [B,X^{\ast}]=0, \ \ [X^{\ast}, X]=Z\overline{Z}, \ \ Z\,^t\!X+XZ=0,  \hspace{2.35cm} \\
\overline{Z}A=\overline{B}\,^t\!B, \ \ A\overline{Z} = B^{\ast}B, \ \
ZB = t_ZB, \ \ \,^t\!(B\overline{B}) - B\overline{B} = \overline{t_Z} A. \hspace{1.5cm}
\end{split}         \right.
\end{equation}
Since $B^{\ast}B= A\overline{Z} = \,^t\!Z\overline{Z}-Z\overline{Z}= \,^t\!Z\overline{Z} - [X^{\ast},X]$, by taking trace we get $|\!|B|\!|^2= |\!|Z|\!|^2$. Therefore, if $B=0$ or $Z=0$, we get $B=Z=A=0$ and the metric $g$ will be K\"ahler and flat. Similarly, if $A=0$, then by $\overline{B}\,^t\!B = \overline{Z}A$ we get $B=0$ hence $Z=0$. So in the following let us assume that $A,Z,B$ are all non-trivial.

Since $BZ=0$, $Z$  must be singular. Denote by $r$ its rank, we have $ 1\leq r<n-1$. Choose unitary basis $\{ e_2, \ldots , e_n\}$ so that the first $r$ rows of $Z$ are linearly independent while the other rows are all zero, and write $B$ in the same block form:
$$ Z = \left[ \begin{array}{cc} x & y \\ 0 & 0 \end{array} \right] , \ \ \ B = \left[ \begin{array}{cc} 0 & u \\ 0 & w \end{array} \right] $$
where $x$ is an $r\times r$ matrix, and the first column blocks of $B$ vanish because $BZ=0$. Since $\overline{Z}A=\overline{B}\,^t\!B$, the lower right corner gives us $\overline{w}\,^t\!w=0$ hence $w=0$. Therefore $B$ has only the upper right corner thus $B\overline{B}=0$, which leads to $\overline{t_Z}A=0$. Therefore we must have $t_Z=0$ as $A$ is assumed to be non-trivial.

Now we claim that the block $x$ must also be zero. Since $w=0$, by comparing the upper left corner of $A\overline{Z}=B^{\ast }B$, we get $(\,^t\!x-x)\overline{x}=0$, so $\mbox{tr}(x\overline{x})=|\!|x|\!|^2$. On the other hand, since $Z\overline{Z}=[X^{\ast},X]$, we have $\mbox{tr}(x\overline{x})=\mbox{tr}(Z\overline{Z})=0$, hence $x=0$.

Note that the $r\times (n-1-r)$ matrix $y$ has rank $r$, so after a unitary change in the last $(n-1-r)$ elements of $e_i$, we may assume that $y=(z,0)$, where $z$ is an $r\times r$ non-singular matrix. Write $u=(b,c)$ where $b$ is $r \times r$, then by $A\overline{Z}=B^{\ast}B$ we know that we must have $c=0$, hence the matrices now take the form
$$
Z = \left[ \begin{array}{ccc} 0 & z & 0 \\ 0 & 0 & 0 \\ 0 & 0 & 0 \end{array} \right] , \ \ \ A = \left[ \begin{array}{ccc} 0 & -z & 0 \\ \,^t\!z & 0 & 0 \\ 0 & 0 & 0 \end{array} \right], \ \ \ B = \left[ \begin{array}{ccc} 0 & b & 0 \\ 0 & 0 & 0 \\ 0 & 0 & 0 \end{array} \right] ,
$$
where $z$ and $b$ are $r\times r$ matrices with $z$ non-singular. The equations in (\ref{BTP-final}) implies that
\begin{equation}  \label{matrixeq}
\overline{z}\,^t\!z = \overline{b}\,^t\!b , \ \ \ ^t\!z \overline{z} = b^{\ast} b, \ \ \ b\,^t\!z= z\,^t\!b,
\end{equation}
and via the non-singularity of $b$ that
$$ X = \left[ \begin{array}{ccc} 0 & 0 & 0 \\ 0 & 0 & 0 \\ 0 & 0 & x \end{array} \right] , \ \ \ \  [x^{\ast } , x]=0. $$
All the solutions to the system of matrix equations (\ref{matrixeq}) are given by the following elementary lemma, and we include its proof here for readers convenience:

\begin{lemma} \label{lemma9}
Two non-singular square matrices $z$ and $b$ satisfy (\ref{matrixeq}) if and only if
\begin{equation*} \label{matrixeq-solution}
b=USV^{\ast}, \ \ \ z= USW\,^t\!V,
\end{equation*}
where $S>0$ is diagonal, $U$, $V$, $W$ are unitary, with $\,^t\!W=W$ and $WS=SW$.
\end{lemma}

\begin{proof}[{\bf Proof of Lemma \ref{lemma9}:}]
Since  $\overline{z}\,^t\!z = \overline{b}\,^t\!b$ is positive definite, it equals to $US^2U^{\ast}$ for some unitary matrix $U$ and a positive definite  diagonal matrix $S$. This means that $V^{\ast} = S^{-1}U^{-1}\overline{b}$ and $Q=S^{-1}U^{-1}\overline{z}$ are both unitary. Plug $b=\overline{U}S\,^t\!V$ and $z=\overline{U}S\overline{Q}$ into the other two equations of (\ref{matrixeq}), we see that the unitary matrix $W=Q\overline{V}$ is symmetric and commutes with $S^2$, hence with $S$. This completes the proof of the lemma.
\end{proof}

Now let us continue with the proof of Proposition \ref{prop5}. If we group $S$ into blocks with equal diagonal values, then $W$ is block-diagonal, with each block being symmetric and unitary. Write $S_{ij}=\lambda_i \delta_{ij}$, $2\leq i,j\leq r+1$. Clearly, we can choose the unitary basis $e$ so that
$$ Z = \left[ \begin{array}{ccc} 0 & SW & 0 \\ 0 & 0 & 0 \\ 0 & 0 & 0 \end{array} \right] , \ \ \ X = \left[ \begin{array}{ccc} 0 & 0 & 0 \\ 0 & 0 & 0 \\ 0 & 0 & x \end{array} \right] , \ \ \   Y=X+B = \left[ \begin{array}{ccc} 0 & S & 0 \\ 0 & 0 & 0 \\ 0 & 0 & x \end{array} \right]  $$
where $x$ is diagonal. The structure equation now takes the form
\begin{equation*}
\left\{ \begin{split} d\varphi_1 = 0, \hspace{6.4cm} \\
d\varphi_i =  \lambda_i \overline{\varphi}_1\varphi_i - \lambda_i \varphi_1 \sum_{j=2}^{r+1} \overline{W_{ij}}\varphi_j ,  \ \ \ \ \ 2\leq i\leq r+1, \hspace{0.1cm} \\
d \varphi_i = 0, \ \ \ \ \ \ \ \ \  r+2\leq i\leq 2r+1,   \hspace{2.45cm}\\
d\varphi_i = (\overline{a}_i\overline{\varphi}_1 - a_i \varphi_1)\varphi_i, \ \ \ \ \ 2r+2\leq i\leq n,   \hspace{1.2cm}
\end{split} \right.
\end{equation*}
which is (\ref{BTPv0}) given in \S 1 before. Here each $a_i$ is complex,  $S=\mbox{diag} \{ \lambda_2 , \ldots , \lambda_{r+1}\} > 0$ and $W$ is any unitary symmetric $r\times r$ matrix satisfying $WS=SW$. This completes the proof of Proposition \ref{prop5}.
\end{proof}

Note that all three cases in Proposition \ref{prop5} are non-K\"ahler and not Chern flat, with the first two cases being not balanced while the third case being balanced. Also, since a Hermitian metric is Bismut K\"ahler-like if and only if it is both {\em BTP} and pluriclosed by \cite{ZZ-pluriclosed}, so from (\ref{SKT-XYZ}) we see that (\ref{BTPv1}) is pluriclosed hence is Bismut K\"ahler-like, while (\ref{BTPv2}) and (\ref{BTPv0}) are not plurilcosed hence not Bismut K\"ahler-like.

\vspace{0.3cm}

\vs

\noindent\textbf{Acknowledgments.} The second named author would like to thank Bo Yang and Quanting Zhao for their interests and/or helpful discussions.

\vs

\end{document}